\documentclass[11pt]{amsart}

\usepackage{amsfonts,epsfig}
\usepackage{latexsym}
\usepackage{amssymb}
\usepackage{amsmath}
\usepackage{amsthm}
\usepackage{graphics}
\usepackage[all]{xy}
\usepackage[T2A]{fontenc}
\usepackage{multirow}
\usepackage{array, booktabs, ctable}
\usepackage{hyperref}
\usepackage{color}
\usepackage{mathtools}

\newtheorem{defn}{Definition}[section]

\DeclarePairedDelimiter{\ceil}{\lceil}{\rceil}
\DeclarePairedDelimiter{\floor}{\lfloor}{\rfloor}

\newtheorem{corollary}[defn]{Corollary}
\newtheorem{lemma}[defn]{Lemma}

\newtheorem{thm}[defn]{Theorem}
\newtheorem{theorem}[defn]{Theorem}

\newtheorem{proposition}[defn]{Proposition}

\theoremstyle{definition}

\newcommand{\lmfdbec}[3]{\href{http://www.lmfdb.org/EllipticCurve/Q/#1#2#3}{{\text{\rm#1#2#3}}}}

\newcommand{\Q}{\mathbb Q}
\newcommand{\Z}{\mathbb Z}

\newcommand{\Gal}{\operatorname{Gal}}
\newcommand{\Aut}{\operatorname{Aut}}

\newcommand{\GL}{\operatorname{GL}}

\def\arraystretch{1.5}

\begin{document}



\bibliographystyle{plain}

\title[Torsion Maximal Ableian Extension]{Torsion of Rational Elliptic Curves over the Maximal Abelian Extension of $\Q$}

\author{Michael Chou}

\address{Dept. of Mathematics, Tufts University, Medford, MA, 02155, USA}
\email{michael.chou@tufts.edu} 




\begin{abstract} Let $E$ be an elliptic curve defined over $\Q$, and let $\Q^{ab}$ be the maximal abelian extension of $\Q$. In this article we classify the groups that can arise as $E(\Q^{ab})_{\text{tors}}$ up to isomorphism. The method illustrates techniques for finding explicit models of modular curves of mixed level structure. Moreover we provide an explicit algorithm to compute $E(\Q^{ab})_{\text{tors}}$ for any elliptic curve $E/\Q$.
\end{abstract}

\maketitle


\section{Introduction and Notation}

Let $K$ denote a number field, and let $E$ be an elliptic curve over $K$.  The Mordell-Weil theorem states that the group of $K$-rational points on $E$ form a finitely generated abelian group.  In particular, letting $E(K)$ denote the $K$-rational points on $E$, we have that
$$E(K) \cong \Z^{r_{K}} \oplus E(K)_{\text{tors}}$$
for some finite group $E(K)_{\text{tors}}$, called the torsion of $E$ over $K$.  In fact, due to a theorem of Merel, there is a bound on size of the torsion subgroup that depends only on the degree of $K$ over $\Q$.  Thus there is a finite list of torsion subgroups that appear as $E(K)_{\text{tors}}$ as $K$ varies over number fields of a fixed degree $d$ and $E/K$ varies.  Let $\Phi(d)$ denote the set of torsion subgroups (up to isomorphism) that appear as $E(K)_{\text{tors}}$ for some elliptic curve $E/K$ as $K$ ranges over all number fields of a fixed degree $d$ over $\Q$.  In particular, in \cite{mazur1} Mazur determined $\Phi(1)$.  Not many other values of $\Phi(d)$ have been determined. The set $\Phi(2)$ was classified by Kamienny (\cite{kamienny} 1992), Kenku, Momose (\cite{kenkumomose} 1988), and the set $\Phi(3)$ was classified by Derickx and independently by Etropolski, Morrow, Zureick-Brown (\cite{cubicclass} 2016).

Classifying torsion subgroups of elliptic curves over number fields is equivalent to classifying points on the modular curves $X_1(M,N)$ defined over these number fields. Thus, the classification of $\Phi(d)$ involves determining all such modular curves with $K$-rational points for any number field $K$ of degree $d$ over $\Q$.

One may also ask a more refined question. Let $\Phi_{\Q}(d)$ denote the set of torsion subgroups (up to isomorphism) that appear as $E(K)_{\text{tors}}$ for some elliptic curve $E/\Q$ as $K$ ranges over all number fields of degree $d$ over $\Q$.  Notice that necessarily $\Phi_{\Q}(d) \subseteq \Phi(d)$ since we are restricting the set of elliptic curves we are considering.  Of course, $\Phi_{\Q}(1) = \Phi(1)$. The sets $\Phi_{\Q}(2)$ and $\Phi_{\Q}(3)$ were determined by Najman (\cite{najman} 2015). A subset of $\Phi_{\Q}(4)$, namely $E(K)_{\text{tors}}$ for $[K:\Q]=4$ and $K/\Q$ abelian, was classified by the author (\cite{chou} 2016), and $\Phi_{\Q}(4)$ has been determined by Gonz\'{a}lez-Jim\'{e}nez and Najman (\cite{gonznajman}). For a more in depth summary of what is known about torsion of elliptic curves over number fields of a fixed degree $d$ see for instance the introduction of \cite{chou}.

In the setting of modular curves, these torsion subgroups can be viewed as $K$-rational points for some $[K:\Q]=d$ on $X_1(M,N)$ whose image under the $j$-map is in $\Q$. These elliptic curves obtain the torsion structure $\Z/M\Z \oplus \Z/N\Z$ over a degree $d$ number field as they correspond to a $K$-rational point on $X_1(M,N)$, but their $j$-invariants are in $\Q$, as each elliptic curve can be defined over $\Q$. 

One can also consider torsion over an infinite extension $L$ of $\Q$. For a fixed algebraic extension $L$ of $\Q$ let $\Phi_{\Q}(L)$ denote the set of torsion subgroups $E(L)_{\text{tors}}$ up to isomorphism  that appear as $E/\Q$ varies. The Mordell-Weil theorem no longer applies, and so a priori it is not guaranteed that the size of $E(L)_{\text{tors}}$ is finite, let alone uniformly bounded as $E$ varies.  Even so, in certain infinite extensions the number of torsion points is finite and, in fact, uniformly bounded as $E$ varies. Fujita determined $\Phi_{\Q}(\Q(2^{\infty}))$ where $\Q(2^{\infty})$ is the compositum of all degree 2 extensions of $\Q$, i.e. $\Q(2^{\infty}) := \mathbb{Q}(\{\sqrt{m} : m \in \mathbb{Z} \})$.

\begin{theorem}[Fujita \cite{fujita}, Theorem 2]\label{fujita}
\[
\Phi_{\Q}(\Q(2^{\infty})) = \left\{
\begin{array}{lr}
\Z/N_1\Z, & N_1 = 1, 3, 5, 7, 9, 15, \\
\Z/2\Z \times \Z/2N_2\Z, & N_2 = 1, 2, 3, 4, 5, 6, 8, \\
\Z/4\Z \times \Z/4N_4\Z, & N_4 = 1,2,3,4, \\
\Z/3\Z \times \Z/3\Z, & \\
\Z/6\Z \times \Z/6\Z, & \\
\Z/8\Z \times \Z/8\Z.
\end{array} \right\}
\]
\end{theorem}

Torsion over a similar infinite extension, $\Q(3^{\infty})$, the compositum of all cubic number fields, was studied by Daniels, Lozano-Robledo, Najman, and Sutherland in \cite{3infinity}.  They classify $\Phi_{\Q}(\Q(3^{\infty}))$. Moreover they determine which of these torsion structures appear infinitely often and which appear for only finitely many isomorphism classes of elliptic curves.

Here is some notation that will be used throughout the paper: $E[p^{\infty}]$ denotes torsion points of order a power of $p$ and $\Q^{ab}$ denotes the maximal abelian extension of $\Q$. By the Kronecker-Weber theorem we have that $\Q^{ab} = \Q(\{ \zeta_n : n \in \Z^{+} \})$ where $\zeta_n$ denotes a primitive $n$-th root of unity.

Given an abelian variety $A/\Q$, the torsion subgroup of $A(\Q^{ab})$ is finite (this is due to a theorem of Ribet \cite{ribet}).  Thus, one can ask if there is a uniform bound for the size of such a torsion subgroup or whether there are possibly infinitely many torsion structures that appear. If we restrict to genus one abelian varietes, we prove there are only finitely many groups that appear as $E(\Q^{ab})_{\text{tors}}$ for any elliptic curve $E/\Q$. In fact, we completely determine $\Phi_{\Q}(\Q^{ab})$.

\begin{theorem}\label{classification}
Let $E/\mathbb{Q}$ be an elliptic curve.  Then $E(\Q^{ab})_{\textup{tors}}$ is isomorphic to one of the following groups:
$$
\begin{array}{lr}
\Z/N_1\Z, & N_1 = 1, 3, 5, 7, 9, 11, 13, 15, 17, 19, 21, 25, 27, 37, 43, 67, 163, \\
\Z/2\Z \times \Z/2N_2\Z, & N_2 = 1, 2, \ldots, 9,\\
\Z/3\Z \times \Z/3N_3\Z, & N_3 = 1, 3,\\
\Z/4\Z \times \Z/4N_4\Z, & N_4 =  1, 2, 3, 4,\\
\Z/5\Z \times \Z/5\Z, & \\
\Z/6\Z \times \Z/6\Z, &  \\
\Z/8\Z \times \Z/8\Z. & \\
\end{array}
$$
Each of these groups appear as $E(\Q^{ab})_{\textup{tors}}$ for some elliptic curve over $\Q$.
\end{theorem}

A uniform bound on the size of $E(\Q^{ab})_{\text{tors}}$ for all elliptic curves $E/\Q$ is an easy corollary of the classification.

\begin{corollary}\label{sizebound}
Let $E/\Q$ be an elliptic curve.  Then $\#E(\Q^{ab})_{\textup{tors}} \leq 163$.  This bound is sharp, as the curve \lmfdbec{26569}{a}{1} has a point of order $163$ over $\Q^{ab}$.
\end{corollary}

In Section \ref{sec:isogeny} we discuss what is known about isogenies of elliptic curves over $\Q$. We then discuss the intimate connection between isogenies and torsion points over $\Q^{ab}$. In Section \ref{sec:bounding} we use the results from section \ref{sec:isogeny} to prove bounds on the group $E(\Q^{ab})_{\text{tors}}$ based on the isogenies $E$ has over $\Q$. In Section \ref{sec:class} we further refine the bounds to eliminate the possibility of any group not appearing in Theorem \ref{classification}. In Section \ref{sec:algorithm} we construct an algorithm to determine $E(\Q^{ab})_{\text{tors}}$ for any elliptic curve $E/\Q$. Finally, Section \ref{sec:examples} has, for each subgroup $T$ appearing in Theorem \ref{classification}, an example of an elliptic curve $E/\Q$ such that $E(\Q^{ab})_{\text{tors}} \cong T$ completing the proof of Theorem \ref{classification}. We use Cremona labels for our elliptic curves and more information on each curve can be found on the LMFDB \cite{lmfdb}.

\noindent \textbf{Acknowledgements.} This work was inspired by \cite{gonzlozano}, and the author would like to thank \'{A}lvaro Lozano-Robledo for his invaluable guidance and input. The author would like to thank Jeremy Rouse and David Zureick-Brown for their helpful advice concerning modular curves. The author would also like to thank Pete Clark and Drew Sutherland for their interest and support for this project. Also, thanks to Harris Daniels and Filip Najman for their comments and suggestions.

\section{Isogenies}\label{sec:isogeny}

In the rest of the paper, when we refer to an isogeny, we will mean a cyclic $\Q$-rational isogeny. The classification of $\Q$-rational $n$-isogenie is an integral part of the classification of torsion of elliptic curves $E/\Q$ over $\Q^{ab}$.

\begin{thm}[Fricke, Kenku, Klein, Kubert, Ligozat, Mazur, and Ogg, among others]\label{isogoverQ}
If $E/\mathbb{Q}$ has an $n$-isogeny, $n \leq 19$ or $n \in \{21,25,27,37,43,67,163\}$.  If $E$ does not have complex multiplication, then $n \leq 18$ or $n \in \{21, 25, 37\}$.  
\end{thm}

See \cite{lozanorobledo1}, Section 9 for a more detailed discussion of this theorem. Moreover, there is a detailed bound on the number of $\Q$-isogenies an elliptic curve can have. The following Theorem is from \cite{kenku}, combining Theorem 2 and the surrounding discussion.

\begin{theorem}[Kenku, \cite{kenku}]\label{8qisogs}
There are at most eight $\Q$-isomorphism classes of elliptic curves in each $\Q$ isogeny class.

Let $C_{p}(E)$ denote the number of distinct $\Q$-rational cyclic subgroups of order $p^{n}$ for any $n$ of $E$. Then, we have the following table for bounds on $C_{p}$ for any elliptic curve over $\Q$
$$
\begin{array}{|c|ccccccccccccc|}
\hline p & 2 & 3 & 5 & 7 & 11 & 13 & 17 & 19 & 37 & 43 & 67 & 163 & \text{else} \\ \hline
 C_{p} & 8 & 4 & 3 & 2 & 2 & 2 & 2 & 2 & 2 & 2 & 2 & 2 & 1 \\ \hline
\end{array}
$$

In particular, fix a $\Q$-isogeny class and a representative $E$ of that class.  

\begin{itemize}

\item If $C_{p}(E) = 2$ for some prime $p \geq 11$, then $C_{q}(E)=1$ for all other primes.  So $C(E)=2$.

\item If $C_7(E)=2$, then $C_5(E)=1$ and either $C_3(E) \leq 2$ and $C_2(E) = 1$ or $C_3(E)=1$ and $C_2(E) \leq 2$.  All these yield $C(E) \leq 4$.

\item If $C_5(E)=3$, then $C_p(E)=1$ for all primes $p \neq 5$.

\item If $C_5(E)=2$, then either $C_3(E) \leq 2$ and $C_2(E)=1$ or $C_3(E)=1$ and $C_2(E) \leq 2$.  Hence $C(E) \leq 4$.

\item If $C_3(E)=4$, then there exists a representative of the class of $E$ with a $\Q$-rational cyclic subgroup of order $27$, and $C_2(E)=1$ so $C(E) \leq 4$.

\item If $C_3(E)=3$, then $C_2(E) \leq 2$ so that $C(E) \leq 6$.

\item If $C_3(E) \leq 2$, then $C_2(E) \leq 4$ so that $C(E) \leq 8$.

\end{itemize}

Note the fact that $C(E)=8$ is possible only if $C_2(E)=8$ or $C_3(E)=2$ and $C_2(E)=4$.

\end{theorem}

The first connection between isogenies and points over $\Q^{ab}$ is shown in the following Lemma.

\begin{lemma}\label{isogimpliespoint}
If $E/\Q$ has an $n$-isogeny defined over $\Q$ then $E(\Q^{ab})$ has a point of order $n$.
\end{lemma}
\begin{proof}
Let $\varphi$ denote the $n$-isogeny over $\Q$. Then $\ker(\varphi) = \langle P \rangle$ for some point $P \in E(\overline{\Q})$ of order $n$ such that $\langle P \rangle^{\sigma} = \langle P \rangle$ for all $\sigma \in \Gal(\overline{\Q}/\Q)$. This induces a character
$$\psi : \Gal(\overline{\Q}/\Q) \rightarrow \left(\Z/n\Z\right)^{\times}$$
defined by $\sigma \mapsto a \bmod n$ where $a$ is given by $\sigma(P) = aP$. The kernel of $\psi$ is precisely $\Gal(\overline{\Q}/\Q(P))$, and thus we have that $\Gal(\Q(P)/\Q)$ is isomorphic to a subgroup of $\left(\Z/n\Z\right)^{\times}$, and hence abelian. Therefore, $P \in E(\Q^{ab})$.
\end{proof}

Given an elliptic curve $E/\Q$, due to Ribet's theorem we know that there exists $m,n \in \Z^{\geq 0}$ such that $E(\Q^{ab})_{\text{tors}} \cong \Z/m\Z \times \Z/mn\Z$.  We wish to understand what possible $m$ and $n$ can occur together.  

In regards to the values of $m$, normally one could use an argument via the Weil-pairing which implies that our field must contain $\zeta_m$, however this is not very restrictive when looking at torsion over $\Q^{ab}$.  Instead, we have the following Theorem:

\begin{theorem}[Gonz\'{a}lez-Jim\'{e}nez, Lozano-Robledo; \cite{gonzlozano}, Theorem 1.1]\label{fulltorsionabelian}
Let $E/\Q$ be an elliptic curve.  If there is an integer $n \geq 2$ such that $\Q(E[n]) = \Q(\zeta_n)$, then $n=2,3,4, \text{ or }5$.  More generally, if $\Q(E[n])/\Q$ is abelian, then $n = 2,3,4,5,6, \text{ or }8$.  Moreover, $\Gal(\Q(E[n]/\Q)$ is isomorphic to one of the following groups:\\
\begin{center}
\begin{tabular}{|c||c|c|c|c|c|c|} \hline
$n$ & $2$ & $3$ & $4$ & $5$ & $6$ & $8$ \\ \hline
\multirow{4}{*}{$\Gal(\Q(E[n]/\Q)$} & $\{0 \}$ & $\Z/2\Z$ & $\Z/2\Z$ & $\Z/4\Z$ & $(\Z/2\Z)^2$ & $(\Z/2\Z)^4$ \\
 & $\Z/2\Z$ & $(\Z/2\Z)^2$ & $(\Z/2\Z)^2$ & $\Z/2\Z \times \Z/4\Z$ & $(\Z/2\Z)^3$ & $(\Z/2\Z)^5$ \\
 & $Z/3\Z$ & & $(\Z/2\Z)^3$ & $(\Z/4\Z)^2$ & & $(\Z/2\Z)^6$ \\
 & & & $(\Z/2\Z)^4$ & & & \\ \hline
\end{tabular}
\end{center}
Furthermore, each possible Galois group occurs for infinitely many distinct j-invariants.
\end{theorem}

In fact, if $E(\Q^{ab})_{\text{tors}} \cong \Z/m\Z \times \Z/mn\Z$, both values $m$ and $n$ are controlled primarily by isogenies. For instance, in the proof of Theorem \ref{fulltorsionabelian}, Gonz\'{a}lez-Jim\'{e}nez and Lozano-Robledo make use of a key corollary relating full-$p$-torsion over $\Q^{ab}$ to $\Q$-rational $p$-isogenies:


\begin{corollary}[Gonz\'{a}lez-Jim\'{e}nez, Lozano-Robledo; \cite{gonzlozano}, Corollary 3.9]\label{CPlowerbound}
Let $E/\Q$ be an elliptic curve, let $p > 2$ be a prime, and suppose that $\Q(E[p])/\Q$ is abelian.  Then, the $\Q$-isogeny class of $E$ contains at least three distinct $\Q$-isomorphism classes, and $C_{p}(E) \geq 3$.  In particular $p \leq 5$.
\end{corollary}

In particular, the proof of Corollary 2.4 in \cite{gonzlozano} shows that for all  $p > 2$ if $\Q(E[p])/\Q$ is abelian for some $E/\Q$, then $E$ has two independent $p$-isogenies over $\Q$. Note that the converse is also true.

\begin{lemma}\label{lem:Epabelian}
Let $E/\Q$ be an elliptic curve, let $p$ be a prime, and suppose that $E$ has two distinct $p$-isogenies over $\Q$. Then $\Q(E[p])/\Q$ is abelian.
\end{lemma}
\begin{proof}
Since $E$ has two independent $p$-isogenies, there exists a basis $\{ P, Q \}$ of $E[p]$ so that the image of $\rho_{E,p}$ is contained in a split Cartan subgroup of $\GL(2,p)$. Now, since $\ker \rho_{E,p} = \Gal(\overline{\Q}/\Q(E[p]))$, it follows that
\[
\Gal(\Q(E[p])/\Q) = \Gal(\Q(P,Q)/\Q) \cong \rho_{E,p}(\Gal(\overline{\Q}/\Q))
\]
which is contained in split Cartan subgroup and thus abelian. 
\end{proof}

Note that we can see in (\cite{gonzlozano}, Table 1) a complete table showing which CM curves can have $\Q(E[n])$ abelian for which $n$.  We also have the following lemma to help understand the possible values of $n$.

\begin{lemma}\label{isogeny}
Let $K$ be a Galois extension of $\Q$, and let $E$ an elliptic curve over $\Q$.  If $E(K)_{\textup{tors}} \cong \Z/m\Z \times \Z/mn\Z$, then $E$ has an $n$-isogeny over $\mathbb{Q}$.
\end{lemma}
\begin{proof}
Suppose $E(K)_{\text{tors}} \cong \Z/m\Z \times \Z/mn\Z = \langle P, Q \rangle$ where $P$ has order $m$ and $Q$ has order $mn$.  Then $[m]E(K)_{\text{tors}} = \langle mP, mQ \rangle = \langle mQ \rangle \cong \Z/n\Z$.  Let $\sigma \in \Gal(\overline{\Q}/\Q)$.  Since $K$ is Galois over $\Q$ and $E$ is defined over $\Q$, the action of Galois commutes with multiplication by $m$ and $n$. It follows that $(mQ)^{\sigma} \in E(K)[n] = \langle mQ \rangle$.  Thus, $\langle mQ \rangle$ is a cyclic subgroup of order $n$ that is stable under the action of $\Gal(\overline{\Q}/\Q)$, which implies $E$ has an $n$-isogeny over $\Q$. 
\end{proof}

Thus, the possible values for $m$ (up to a power of 2) and $n$ are controlled by the $\Q$-isogenies of the elliptic curve.

\section{Points of order $2^{n}$}

In order to understand $E(\Q_{ab})[2^{\infty}]$ and it's connection to isognies, we will use the database of Rouse and Zureick-Brown \cite{2adicimage}. First we prove a simple lemma concerning quadratic twists.

\begin{lemma}\label{lemma:twists}
Let $E/\Q$ be an elliptic curve, let $d$ be a square-free integer, and let $E_d$ denote the quadratic twist of $E$ by $d$. Then $E(\Q^{ab})_{\textup{tors}} \cong E_d (\Q^{ab})_{\textup{tors}}$.
\end{lemma}
\begin{proof}
Since $E$ and $E_d$ become isomorphic over $\Q(\sqrt{d})$ and $\Q(\sqrt{d}) \subseteq \Q^{ab}$ for any $d$, the lemma follows immediately.
\end{proof}

Note that the minimal field of definition of the torsion for $E$ and $E_d$ may differ, but by the previous lemma their torsion over $\Q^{ab}$ will always be isomorphic. In particular, when examining elliptic curves with $j$-invariant not equal to 0 or 1728, it suffices to fix a specific curve $E$, and examine $E(\Q^{ab})_{\text{tors}}$.

The following lemma gives a criterion for a point to be halved which will be handy to explicitly compute $\Q(E[2^k])$ for various $k$:

\begin{lemma}[Knapp \cite{knapp}, Theorem 4.2, p. 85]\label{halvingpoint}
Let $K$ be a field of characteristic not equal to 2 or 3, and let $E$ be an elliptic curve over $K$ given by $y^2 = (x-\alpha)(x-\beta)(x-\gamma)$ with $\alpha, \beta, \gamma$ in $K$.  For $P=(x,y) \in E(K)$, there exists a $K$-rational point $Q = (x',y')$ on $E$ such that $[2]Q = P$ if and only if $x- \alpha, x-\beta,$ and $x-\gamma$ are all squares in $K$.  In this case, if we fix the sign of $\sqrt{x-\alpha}, \sqrt{x-\beta},$ and $\sqrt{x-\gamma}$, then $x'$ equals one of the following:
$$\sqrt{x-\alpha}\sqrt{x-\beta} \pm \sqrt{x-\alpha}\sqrt{x-\gamma} \pm \sqrt{x-\beta}\sqrt{x-\gamma} + x$$
or
$$-\sqrt{x-\alpha}\sqrt{x-\beta} \pm \sqrt{x-\alpha}\sqrt{x-\gamma} \mp \sqrt{x-\beta}\sqrt{x-\gamma} + x$$
where the signs are taken simultaneously.
\end{lemma}

In particular, we can prove an nice criterion for an elliptic curve to have a point of order 4 over $\Q^{ab}$, but not full 4-torsion. We will make use of the following proposition describing the Galois group of various degree 4 polynomials.

\begin{proposition}[\cite{conrad} Corollary 4.5]\label{galgrouptype}
Let $f(X) = X^4 + bX^2 + d$ be irreducible in $K[X]$, where $K$ does not have characteristic 2.  Its Galois group over $K$, denoted $G_f$, is $V,\ \Z/4\Z, \text{ or } D_{4}$ according to the following conditions.

\begin{enumerate}

\item If $d \in \left( K^{\times} \right)^2$ then $G_{f} = \Z/2\Z \times \Z/2\Z$.

\item If $d \not\in \left( K^{\times} \right)^2$ and $(b^2 - 4d)d\in \left( K^{\times} \right)^2$ then $G_f = \Z/4\Z$.

\item If $d \not\in \left( K^{\times} \right)^2$ and $(b^2 - 4d)d \not\in \left( K^{\times} \right)^2$ then $G_f = D_4$.

\end{enumerate}

\end{proposition}

We combine this result with Lemma \ref{halvingpoint} to prove the following lemma:

\begin{lemma}\label{z2xz4xzn}

Suppose $E$ is an elliptic curve over $\Q$ with a point of order 4 defined over $\Q^{ab}$ but not full 4-torsion defined over $\Q^{ab}$.  Then, there is a model of $E$ of the form:
\[
E: y^2 = x(x^2 + bx + d)
\]
and either $d$ or $(b^2-4d)d$ is a non-zero perfect square in $\Q$.

\end{lemma}

\begin{proof}

By Lemma \ref{2powertorsion} if $E$ has a point of order 4 defined over $\Q^{ab}$ then $E$ has at least one 2-isogeny over $\Q$. Thus there exists a point $P$ of order 2 defined over $\Q$, and by moving that point to $P = (0,0)$ we obtain a model for $E$ of the form $y^2 = x(x^2 + bx + d)$ with $b,d \in \mathbb{Q}$.

Over $\Q^{ab}$ we obtain a point of order 4, say $Q$.  Suppose first that this point satisfies $2Q = P = (0,0)$.  Writing $E : y^2 = x(x-\alpha)(x-\overline{\alpha})$, Lemma $\ref{halvingpoint}$ says that we have that $\sqrt{-\alpha}$ and $\sqrt{-\overline{\alpha}}$ must be squares in $\Q^{ab}$.  Notice that the minimal polynomial for both $\sqrt{-\alpha}$ and $\sqrt{-\overline{\alpha}}$ is in fact the polynomial
\[
f = x^4 - bx^2 + d.
\]
Thus, we must have that the Galois group of $f$ over $\Q$ is abelian.  Therefore $G_f = \Z/2\Z \times \Z/2\Z \text{ or } \Z/4\Z$.  Now by Proposition $\ref{galgrouptype}$ the two cases above follow.

Suppose instead that our point of order 4 is halving one of the other points of order 2.  Without loss of generality we may assume $2Q = (\alpha,0)$.  By Lemma \ref{halvingpoint} we have that $\alpha$ and $\alpha- \overline{\alpha}$ are squares in $\Q^{ab}$.  Notice that the minimal polynomial of $\sqrt{\alpha}$ is
\[
f = x^4 + bx^2 + d
\]
and the minimal polynomial of $\sqrt{\alpha - \overline{\alpha}}$ is
\[
g = x^4 - (b^2-4d).
\]
Note that $b^2-4d$ is not a square since $\alpha-\overline{\alpha} = \sqrt{b^2-4d}$ was was assumed to not be in $\Q$.

Let $L_1 = \Q(\sqrt{\alpha})$ and $L_2 = \Q(\sqrt{\alpha - \overline{\alpha}})$.  We wish to consider when the composite $L_1 L_2 \subseteq \Q^{ab}$.  For $L_2 \subseteq \Q^{ab}$ we must have $\Gal(L_2/\Q) \cong \Z/2\Z \times \Z/2\Z \mbox{ or } \Z/4\Z$. Applying Proposition \ref{galgrouptype} to the defining polynomial $g$ of $L_2$ gives us conditions for when this happens.

$$\Gal(L_2/\Q) \cong
\begin{cases}
\Z/2\Z \times \Z/2\Z &\mbox{ if } -(b^2-4d) \in \left(\Q^{\times}\right)^2 \\
\Z/4\Z &\mbox{ if } -(b^2-4d) \not\in \left(\Q^{\times}\right)^2 \mbox{ and } -4(b^2-4d)^2 \in \left(\Q^{\times}\right)^2.
\end{cases}
$$
Notice that the second case $\Gal(L_2/\Q) \cong \Z/4\Z$ cannot occur, since if $-4(b^2-4d)^2 \in \left(\Q^{\times}\right)^2$ then $-1 \in \left(\Q^{\times}\right)^2$.

Thus, suppose instead that $\Gal(L_2/\Q) \cong \Z/2\Z \times \Z/2\Z$. Then $-(b^2-4d) = s^2$ for some $s \in \Q$.  Then $\alpha - \overline{\alpha}= \sqrt{b^2 - 4d} = \sqrt{-s^2} = |s|\sqrt{-1}$.  Therefore, since $\alpha - \overline{\alpha} \in L_2$, we have that $\sqrt{-1} \in L_2$.  Recall that $(0,0)$ is halved in the extension where $x^4 - bx^2 + d$ splits.  By assumption, $x^4 + bx^2 + d$ is split in $L_1$, and thus it splits in $L_1 L_2$. Further, we have shown that $\sqrt{-1} \in L_2 \subseteq L_1 L_2$, and therefore $x^4 - bx^2 + d$ must also split.  Thus the point $(0,0)$ is also halved in $L_1 L_2$, and so full 4-torsion is defined over $L_1 L_2$.  Thus, if $L_1 L_2 \subseteq \Q^{ab}$, we contradict our original assumption that $E$ does not have full 4-torsion over $\Q^{ab}$.

\end{proof}

We now prove our proposition relating 2-powered torsion to the isogenies of an elliptic curve.

\begin{proposition}\label{R-DZB}
Let $E/\Q$ be an elliptic curve. Table \ref{table:RDZB} gives the possibilities for $E(\Q^{ab})[2^{\infty}]$, the $2$-powered isogenies attached to each case, and also $C_2(E)$.
\end{proposition}

\begin{table}[!ht]
\begin{tabular}{|c|c|c|} \hline
$E(\Q^{ab})[2^{\infty}]$ & Isogeny Degrees & $C_{2}(E)$\\ \hline
$\{ \mathcal{O} \}$ & 1 & 1 \\ \hline
\multirow{2}{*}{$\Z/2\Z \times \Z/2\Z$} & 1 & 1 \\
	& 1, 2 & 2\\ \hline
\multirow{2}{*}{$\Z/2\Z \times \Z/4\Z$} & 1, 2 & 2\\
	& 1, 2, 4, 4 & 4\\ \hline
\multirow{2}{*}{$\Z/2\Z \times \Z/8\Z$} & 1, 2, 4, 4 & 4\\
	& 1, 2, 4, 4, 8, 8 & 6\\ \hline
\multirow{2}{*}{$\Z/4\Z \times \Z/4\Z$} & 1, 2, 2, 2 & 4\\
	& 1, 2, 4, 4 & 4\\ \hline
$\Z/2\Z \times \Z/16\Z$ & 1, 2, 4, 4, 8, 8, 16, 16 & 8\\ \hline 
\multirow{2}{*}{$\Z/4\Z \times \Z/8\Z$} & 1, 2, 2, 2, 4, 4 & 6\\
	& 1, 2, 4, 4, 8, 8, 8, 8 & 8\\ \hline 
$\Z/4\Z \times \Z/16\Z$ & 1, 2, 2, 2, 4, 4, 8, 8 & 8\\ \hline
$\Z/8\Z \times \Z/8\Z$ & 1, 2, 2, 2, 4, 4, 4, 4 & 8\\ \hline
\end{tabular}
\caption{Possible 2-primary torsion over $\Q^{ab}$}\label{table:RDZB}
\end{table}

\begin{proof}
If $E$ does not have CM, it must be in one of the families given in the Rouse, Zureick-Brown database \cite{2adicimage}.  We compute for each family the 2-powered torsion over $\Q^{ab}$.  We do this as follows: for each family let $G$ be the image of $\rho_{E,32}$, that is $G = \rho_{E,32}(\Gal(\overline{\Q}/\Q))$.  In fact,
$$G = \rho_{E,32}(\Gal(\overline{\Q}/\Q)) \cong\Gal(\Q(E[32])/\Q)$$
since $\ker \rho_{E,32} = \Gal(\overline{\Q}/\Q(E[32]))$.  Then the commutator subgroup $[G,G]$ has fixed field equal to $\Q(E[32])\cap \Q^{ab}$. We fix a $\Z/32\Z$-basis $\{P,Q\}$ of $E[32]$ and identify $G$ with a subgroup of $\GL(2,32)$. We then compute the vectors fixed in $\left(\Z/32\Z\right)^2$ by $[G,G]$ which gives the structure of the points on $E$ defined over $\Q(E[32])\cap \Q^{ab}$, that is the structure of $E(\Q^{ab})[32]$. Here, a vector $[a,b] \in \left(\Z/32\Z\right)^2$ corresponds to a point $aP+bQ \in E[32]$. Since the largest order point found in $E(\Q^{ab})[32]$ has order 16, we see that $E(\Q^{ab})[2^{\infty}] = E(\Q^{ab})[16]$.

For elliptic curves with CM we examine the finitely many $j$-invariants over $\Q$ (see \cite{gonzlozano} Table 1). The table is broken up into quadratic twist families by $j$-invariant. For $j\neq 0,1728$, there are only finitely many quadratic twist families, and so by Lemma \ref{lemma:twists} it suffices to fix a single curve within each family and examine its torsion over $\Q^{ab}$. 

Let $E$ be such a curve. From that table we can see the largest $m$ such that $\Q(E[2^m])$ is abelian, as well as the isogenies each $\Q$-isomorphism class has.  Let $2^n$ denote the largest degree 2-powered isogeny $E$ has. Then from Lemma \ref{isogeny} it follows that $E(\Q^{ab})[2^{\infty}] \subseteq \Z/2^m\Z \times \Z/2^{n+m}\Z$.

Thus, to find the structure of $E(\Q^{ab})[2^{\infty}]$ it simply remains to find the largest 2-powered torsion $E$ has over some abelian number field, up to $2^{n+m}$. We do this by using the division polynomial method. For each $0 < k \leq n+m$ we use Magma \cite{magma} to compute the $(2^k)^{\text{th}}$-division polynomial of $E$, whose roots are the $x$-coordinates of the points of order $2^k$ on $E$. From the $x$-coordinates, we can compute the corresponding $y$-coordinates, and get a list of all points of order $2^k$ on $E$. Now we simply compute the field of definition of these points and check whether each field is abelian or not. If $k_0$ is the first value where no points of order $2^{k_0}$ are defined over an abelian extension, then $E(\Q^{ab})[2^{\infty}] \cong \Z/2^m \times \Z/2^{k_0-1}\Z$.

We run through each quadratic twist family. Once computed we find that we do not gain any new groups nor do we add any new 2-powered isogeny degree combinations to the list originally found for non-CM curves. Note that the code used to do these computations is available on the author's website.

For $j=0$, the cases $y^2=x^3+t^3$, and $y^2=x^3+16t^3$ where $t \in \Q$ are single quadratic twist families, so may be treated as above. For the case $y^2=x^3+s$ with $s \neq t^3, 16t^3$ \cite{gonzlozano} Table 1 already shows that $E(\Q^{ab})[2] = \{ \mathcal{O} \}$ for any of these quadratic twist families.

Similarly for $j=1728$, the cases $y^2=x^3 \pm t^2x$ for $t \in \Q$ are in two separate quadratic twist families, so may be treated as before. Finally, we consider the case $E: y^2= x^3 + sx$, with $s \neq \pm t^2$ for any $t \in \Q$. We see from \cite{gonzlozano} Table 1 that the largest division field that is abelian is $\Q(E[2])$. Suppose that $E$ has a point of order 4 over $\Q^{ab}$, then Lemma \ref{z2xz4xzn} gives that $s$ is a square in $\Q$ or $-4s^2$ is a square in $\Q$, contradicting our assumption that $s \neq \pm t^2$ for any $t \in \Q$. Thus, $E(\Q^{ab})[2^{\infty}] \equiv \Z/2\Z \times \Z/2\Z$. Thus the table above is complete.

\end{proof}

From Table \ref{table:RDZB} we can make a simple observation:

\begin{lemma}\label{2powertorsion}
Let $E/\Q$ be an elliptic curve and suppose that $E(\Q^{ab})[2^{\infty}] \not\cong \{\mathcal{O} \} \text{ or } \Z/2\Z \times \Z/2\Z$.  Then $E$ has at least one $2$-isogeny over $\Q$, that is, $C_{2}(E) \geq 2$.
\end{lemma}

\section{Bounding Torsion}\label{sec:bounding}

We begin with a proposition that bounds $E(\Q^{ab})[p^{\infty}]$ for all primes $p$.

\begin{proposition}\label{pupperbound}
Let $E/\Q$ be an elliptic curve, and $\Q^{ab}$ be the maximal abelian extension of $\Q$.  Then the following table gives a bound on $E(\Q^{ab})[p^{\infty}]$ for all primes $p$, i.e., the $p$-power torsion is contained in the following subgroups:
\begin{center}
\begin{tabular}{|c||c|c|c|c|c|} \hline
$p$ & $2$ & $3$ & $5$ & $7$, $11$, $13$, $17$, $19$,& \emph{else}\\
 & & & & $37$, $43$, $67$, $163$, & \\ \hline
\multirow{2}{*}{$E(\Q^{ab})[p^{\infty}] \subseteq$} &
$\Z/4\Z \times \Z/16\Z$ & \multirow{2}{*}{$\Z/3\Z \times \Z/27\Z$} & \multirow{2}{*}{$\Z/5\Z \times \Z/25\Z$} & \multirow{2}{*}{$\Z/p\Z$} & \multirow{2}{*}{$\{\mathcal{O}\}$}\\
& $\Z/8\Z \times \Z/8\Z$ & & & & \\ \hline
\end{tabular}
\end{center}
\end{proposition}
\begin{proof}
Note that $E(\overline{\Q})[p^{n}] \cong \Z/p^n\Z \times \Z/p^n\Z$, and thus $E(\Q^{ab})[p^n] \subseteq \Z/p^n\Z \times \Z/p^n\Z$ for any $n$.  However, by Theorem \ref{fulltorsionabelian}, if $\Q(E[p^n])$ is abelian, then $p = 2,\; 3, \text{ or } 5$, and therefore for any prime except 2, 3, and 5, we must have that $E$ does not have full $p$ torsion defined over $\Q^{ab}$.  Thus, if $p > 5$ then $E(\Q^{ab})[p^{\infty}] \subseteq \Z/p^n\Z$ for some $n$.  However, since $\Q^{ab}$ is a Galois extension of $\Q$, Lemma \ref{isogeny} shows that $E(\Q^{ab})[p^{\infty}] \subseteq \Z/p\Z$ for $p =$ 7, 11, 13, 17, 19, 37, 43, 67, and 163, and for all other primes $l$ larger than 5, $E(\Q^{ab})[l^{\infty}] \cong \{ \mathcal{O} \}$.

For the prime $p=2$, we simply refer to Table \ref{table:RDZB}.

For the prime $p=3$, first notice that $E$ cannot have full 9-torsion over $\Q^{ab}$ because of Theorem \ref{fulltorsionabelian}.  Thus, $E(\Q^{ab})[3^{\infty}] \subseteq \Z/3\Z \times \Z/3^e\Z$ for some natural number $e$.  By Lemma \ref{isogeny} if $E(\Q^{ab})[3^{\infty}] \cong \Z/3\Z \times \Z/3^e\Z$ then $E$ has a $3^{e-1}$ isogeny.  By Theorem \ref{isogoverQ}, the largest 3-power degree rational isogeny is 27, and so $e-1 \leq 3$, i.e. $E(\Q^{ab})[3^{\infty}] \subseteq \Z/3\Z \times \Z/81\Z$.  However, suppose that in fact $E(\Q^{ab})[3^{\infty}] \cong \Z/3\Z \times \Z/81\Z$.  Then by the above argument, $E$ has a rational 27-isogeny.  However, the only elliptic curves over $\Q$ that have a 27-isogeny are CM curves (those with $j$-invariant $-2{15} \cdot 3 \cdot 5^3$).  By Table 1 in \cite{gonzlozano} we see that such a curve does not have $\Q(E[n])$ abelian for any $n \geq 2$.  Thus, $E(\Q^{ab})[3^{\infty}] \subseteq \Z/3\Z \times \Z/27\Z$.

For the prime $p=5$, first notice that $E$ cannot have full 25-torsion over $\Q^{ab}$ because of Theorem \ref{fulltorsionabelian}.
By an identical argument as in the $p=3$ case, we have that $E(\Q^{ab})[5^{\infty}] \cong \Z/5\Z \times \Z/125\Z$, since the largest 5-power degree rational isogeny is 25. However, suppose that $E(\Q^{ab})[5^{\infty}] \cong \Z/5\Z \times \Z/25\Z$.  Consider the Galois representation $\rho_{E,25}: \Gal(\overline{\Q}/\Q) \rightarrow \Aut(E[25]) \cong \GL(2,25)$.  Let $G$ denote the image of $\rho_{E,25}$.  Since full 5-torsion is defined over $\Q^{ab}$, Corollary \ref{CPlowerbound} says there is a basis of $E[5]$ such that $G \bmod 5$ is contained in a split Cartan subgroup of $\GL(2,5)$.  Thus, we have that
$$G \leq \mathcal{G} := \left\{ \begin{bmatrix} a & 5b \\ 5c & d \end{bmatrix} : a,d \in \left(\Z/25\Z\right)^{\times} , b,c \in \Z/5\Z \right\}.$$

Now, let $H$ denote the image $\rho_{E,25}(\Gal(\overline{\Q}/\Q(E[25])\cap\Q^{ab}))$.  Notice that $H = [G,G]$ the commutator subgroup of $G$.  Since $E$ has a point of order 25 over $\Q^{ab}$, we have that $H$ must fix a vector of order 25 in $\left(\Z/25\Z\right)^2$, and therefore $[G,G]$ must also fix a vector of order 25 in $\left(\Z/25\Z\right)^2$.

By the Weil-pairing the image of $\rho_{E,25}$ must have determinant equal to the full group $\left(\Z/25\Z\right)^{\times}$, and therefore $G$ must be a subgroup of $\mathcal{G}$ with full determinant, and whose commutator subgroup fixes a vector of order 25 in $\left(\Z/25\Z\right)^2$.  Using Magma \cite{magma} we can compute all such subgroups of $\GL(2,25)$, and further we can also compute, given a subgroup of $\GL(2,25)$, the isogenies of an elliptic curve associated with that image.

We thus compute that in fact all subgroups of $\mathcal{G}$ with the described properties all yield a 25-isogeny, and thus any elliptic curve with such an image must in fact have a 25-isogeny.  However, since full 5-torsion was defined over $\Q^{ab}$, Corollary \ref{CPlowerbound} gives two isogenies of degree 5 and thus it is impossible for $E$ to have a 25-isogeny, otherwise $C_5(E) = 4$ contradicting Theorem \ref{8qisogs}.  Thus, $E(\Q^{ab})[5^{\infty}] \subseteq \Z/5\Z \times \Z/5\Z$.
\end{proof}

To prove bounds on the structure of $E(\Q^{ab})_{\text{tors}}$ We will need a lemma about full 6-torsion from \cite{gonzlozano}:

\begin{lemma}[\cite{gonzlozano}, Lemma 3.12]\label{full6torsion}
Let $E/\Q$ be an elliptic curve.  If $\Q(E[6])/\Q$ is abelian, then $\Gal(\Q(E[2])/\Q) \cong \Z/2\Z$.
\end{lemma}

As has been noted upon in Section \ref{sec:isogeny}, the structure of torsion over $\Q^{ab}$ is closely tied to the $\Q$-isogenies an elliptic curve has. We now prove bounds on $E(\Q^{ab})_{\text{tors}}$ based on these isogenies.

\begin{proposition}\label{prop:noisogenies}
Let $E/\Q$ be an elliptic curve. Suppose $E$ has no non-trivial isogenies over $\Q$. Then $E(\Q^{ab})_{\emph{tors}} \cong \{\mathcal{O}\} \text{ or } \Z/2\Z \times \Z/2\Z$.
\end{proposition}
\begin{proof}
By Lemma \ref{isogeny} it follows that $E(\Q^{ab})_{\text{tors}} \cong \Z/m\Z \times \Z/m\Z$ for some $m \geq 1$. However, Corollary \ref{CPlowerbound} implies that $m$ is a power of 2. Combining that with Lemma \ref{2powertorsion} shows that $m=1 \text{ or } 2$.
\end{proof}

\begin{proposition}\label{torsionisZpZ}
Let $E/\Q$ be an elliptic curve, let $p = 11, 17, 19, 37, 43, 67, \text{ or } 163$ and suppose that $E$ has a $p$-isogeny. Then $E(\Q^{ab})_{\emph{tors}} \cong \Z/p\Z$.
\end{proposition}
\begin{proof}
By Lemma \ref{isogimpliespoint} we have that $E(\Q^{ab})_{\text{tors}} \supseteq \Z/p\Z$. For these values of $p$, note that there are no rational isogenies of degree divisible by $p$ besides isogenies of degree exactly $p$, and therefore by Lemma \ref{isogeny} it follows that $E(\Q^{ab})_{\text{tors}} \subseteq \Z/m\Z \times \Z/m\Z \times \Z/p\Z$ for some $m$ with $(m,p)=1$. However, from Theorem \ref{8qisogs} it follows that $E$ has no other rational isogenies. Thus Corollary \ref{CPlowerbound} implies that $m$ is a power of 2. Combining that with Lemma \ref{2powertorsion} shows that $m = 1 \text{ or } 2$.

For any given $p$ in this list there are only finitely many $j$-invariants of elliptic curves having a $p$-isogeny, as $X_0(p)$ has genus greater than $0$.  Given that these $j$-invariants are not 0 or 1728, by Lemma \ref{lemma:twists} it suffices to fix a representative $E_j$ and compute (via Magma \cite{magma}) that $E_j$ does not have full $2$-torsion defined over $\Q^{ab}$.
\end{proof}

\begin{proposition}\label{prop:2isogeny}
Let $E/\Q$ be an elliptic curve. Suppose $C_p(E) = 1$ for all primes $p \neq 2$. Then $E(\Q^{ab})_{\emph{tors}} = E(\Q^{ab})[2^{\infty}]$ and is contained in either $\Z/4\Z \times \Z/16\Z$ or $\Z/8\Z \times \Z/8\Z$ 
\end{proposition}
\begin{proof}
By Corollary \ref{CPlowerbound} and Lemma \ref{isogeny} it follows that $E(\Q^{ab})_{\text{tors}} = E(\Q^{ab})[2^{\infty}]$. Thus $E(\Q^{ab})_{\text{tors}}$ is one of the groups on Table \ref{table:RDZB} from Proposition \ref{R-DZB}.
\end{proof}

\begin{proposition}\label{prop:3isogeny}
Let $E/\Q$ be an elliptic curve. Suppose $E$ has a 3-isogeny and $C_p(E) = 1$ for all primes $p>3$. Then 
$$E(\Q^{ab})_{\emph{tors}} \subseteq
\begin{cases}
\Z/2\Z \times \Z/2N_2\Z, & N_2 = 12,18, \\
\Z/3\Z \times \Z/27\Z, & \\
\Z/4\Z \times \Z/4\Z, & \text{ or }\\
\Z/6\Z \times \Z/12\Z.
\end{cases}$$
\end{proposition}
\begin{proof}
By Theorem \ref{8qisogs} we have either 
\begin{itemize}
\item $C_3(E) = 4$ and $C_p(E) = 1$ for all primes $p \neq 3$,
\item $C_3(E) = 3$ and $C_2(E) \leq 2$,
\item or $C_3(E) = 2$ and $C_2(E) \leq 4$.
\end{itemize}

Suppose $C_3(E) = 4$ and $C_p(E) = 1$ for all primes $p \neq 3$. Then by Proposition \ref{pupperbound} we know that $E(\Q^{ab})[3^{\infty}] \subseteq \Z/3\Z \times \Z/27\Z$. Suppose that $E$ has full 2-torsion over $\Q^{ab}$ and full 3-torsion over $\Q^{ab}$ and hence full 6-torsion over $\Q^{ab}$.  Then by Lemma \ref{full6torsion} we must have $\Gal(\Q(E[2])/\Q) \cong \Z/2\Z$.  This is possible only if $E$ has a point of order 2 defined over $\Q$, but that would give a Galois stable subgroup of order 2, and hence $C_2(E) \geq 2$ a contradiction.  Therefore by Lemma \ref{2powertorsion} we have that $E(\Q^{ab})[2^{\infty}] \cong \{ \mathcal{O} \}$ and so $E(\Q^{ab})_{\text{tors}} \subseteq \Z/3\Z \times \Z/27\Z$. If $E$ does not have full 3-torsion over $\Q^{ab}$ then $E(\Q^{ab})[3^{\infty}] \subseteq \Z/27\Z$.  If $E(\Q^{ab})[3^{\infty}] \cong \Z/27\Z$ then $E$ has a 27-isogeny and all such curves have CM (for instance see \cite{lozanorobledo1} Table 4).  By Table 1 in \cite{gonzlozano} we see that such a curve does not have full $n$ torsion defined over $\Q^{ab}$ for any $n$.  Thus, $E(\Q^{ab})_{\text{tors}} = E(\Q^{ab})[3^{\infty}] \cong \Z/27\Z$ or $E(\Q^{ab})[3^{\infty}] \subseteq \Z/9\Z$ and $E(\Q^{ab})[2^{\infty}] \subseteq \Z/2\Z \times \Z/2\Z$, which yields $E(\Q^{ab})_{\text{tors}} \subseteq \Z/2\Z \times \Z/18\Z$.

Suppose that $C_3(E)=3$ and $C_2(E) \leq 2$. Then we have that $E(\Q^{ab})[3^{\infty}] \cong \Z/3\Z \times \Z/3\Z \text{ or } \Z/9\Z$. The largest $E(\Q^{ab})[2^{\infty}]$ can be so that $C_2(E) = 2$ is $\Z/2\Z \times \Z/4\Z$.  Therefore, $E(\Q^{ab})_{\text{tors}} \subseteq \Z/6\Z \times \Z/12\Z \text{ or } \Z/2\Z \times \Z/36\Z$.

Suppose that $C_3(E)=2$ and $C_2(E) \leq 4$. Then $E(\Q^{ab})[3^{\infty}] \cong \Z/3\Z$. From Proposition \ref{R-DZB} the largest $E(\Q^{ab})[2^{\infty}]$ can be so that $C_2(E) = 4$ is $\Z/4\Z \times \Z/4\Z$ or $\Z/2\Z \times \Z/8\Z$.  Thus,  $E(\Q^{ab})_{\text{tors}} \subseteq \Z/4\Z \times \Z/12\Z \text{ or } \Z/2\Z \times \Z/24\Z$.
\end{proof}

\begin{proposition}\label{prop:5isogeny}
Let $E/\Q$ be an elliptic curve. Suppose $E$ has a 5-isogeny. Then 
$$E(\Q^{ab})_{\emph{tors}} \subseteq 
\begin{cases}
\Z/5\Z \times \Z/5\Z,\\
\Z/2\Z \times \Z/50\Z,\\
\Z/2\Z \times \Z/30\Z,\text{ or } \\
\Z/2\Z \times \Z/20\Z. 
\end{cases}$$
\end{proposition}
\begin{proof}
By Theorem \ref{8qisogs} we have either
\begin{itemize}
\item $C_5(E) = 3$ and $C_p(E) = 1$ for all primes $p \neq 5$,
\item $C_5(E) = 2$ and $C_3(E) \leq 2$ and $C_2(E) = 1$,
\item or $C_5(E) = 2$ and $C_3(E) = 1$ and $C_2(E) \leq 2$.
\end{itemize}
Suppose $C_5(E) = 3$ and $C_p(E) = 1$ for all primes $p \neq 5$. By Corollary \ref{CPlowerbound} we see that $E$ does not have full 3-torsion over $\Q^{ab}$. By Lemma \ref{2powertorsion} we have $E(\Q^{ab})[2^{\infty}] \cong \{\mathcal{O}\} \text{ or } \Z/2\Z \times \Z/2\Z$. If $E(\Q^{ab})[5^{\infty}] \cong \Z/5\Z \times\Z/5\Z$, then $E(\Q^{ab})[2^{\infty}] \cong \{\mathcal{O}\}$, since otherwise $E$ would have full 10-torsion over $\Q^{ab}$, contradicting Theorem \ref{fulltorsionabelian}. Thus if $E(\Q^{ab})[5^{\infty}] \cong \Z/5\Z \times \Z/5\Z$, then $E(\Q^{ab})_{\text{tors}} = E(\Q^{ab})[5^{\infty}] \cong \Z/5\Z \times \Z/5\Z$. If $E$ does not have full 5-torsion over $\Q^{ab}$, then $E(\Q^{ab})[5^{\infty}] \cong \Z/25\Z$ in order for $C_5(E)=3$. Then $E(\Q^{ab})_{\text{tors}} \cong \Z/25\Z$ or $\Z/2\Z \times \Z/50\Z$.

Suppose $C_5(E) = 2$ and $C_3(E) \leq 2$ and $C_2(E) = 1$. Then $E(\Q^{ab})[5^{\infty}] \cong \Z/5\Z$ and $E(\Q^{ab})[3^{\infty}] \subseteq \Z/3\Z$ by Lemma \ref{isogimpliespoint}. Again by Lemma \ref{2powertorsion} we have $E(\Q^{ab})[2^{\infty}] \cong \{ \mathcal{O} \}$ or $\Z/2\Z \times \Z/2\Z$. Thus, $E(\Q^{ab})_{\text{tors}} \subseteq \Z/2\Z \times \Z/30\Z$.

Suppose $C_5(E) = 2$ and $C_3(E) = 1$ and $C_2(E) \leq 2$. Then again $E(\Q^{ab})[5^{\infty}] \cong \Z/5\Z$. By Table Proposition \ref{R-DZB}, the largest $E(\Q^{ab})[2^{\infty}]$ can be so that $C_2(E) = 2$ is $\Z/2\Z \times \Z/4\Z$. Further, $E$ does not have full torsion of any order prime to 2 by Corollary \ref{CPlowerbound}. Thus, $E(\Q^{ab})_{\text{tors}} \subseteq \Z/2\Z \times \Z/20\Z$.
\end{proof}

\begin{proposition}\label{prop:7isogeny}
Let $E/\Q$ be an elliptic curve. Suppose $E$ has a 7-isogeny. Then either $E(\Q^{ab})_{\emph{tors}} \subseteq \Z/21\Z$ or $\Z/2\Z \times \Z/28\Z$. 
\end{proposition}
\begin{proof}
By Theorem \ref{8qisogs} we have $C_p(E) = 1$ for all primes $p \neq 2,3, 7$ and either $C_3(E) \leq 2$ and $C_2(E) = 1$, or $C_3(E) = 1$ and $C_2(E) \leq 2$.

Suppose $C_3(E) = 2$ and $C_2(E) = 1$. Then $E$ has a 7-isogeny and a 3-isogeny and so $E$ has a 21-isogeny. Since there are only finitely many rational points on $X_0(21)$, there are only a finite number of $j$-invariants for elliptic curves over $\Q$ with a 21-isogeny. We can fix a model for each of these curves and explicitly check that none of these families have full $m$-torsion for any $2 \leq m \leq 8$. Thus by Lemma \ref{isogimpliespoint}, Lemma \ref{isogeny}, and Theorem \ref{fulltorsionabelian} we have $E(\Q^{ab}) \cong \Z/21\Z$.

Suppose instead that $C_3(E) = 1$ and $C_2(E) =2$. Since $C_3(E)=1$, Corollary \ref{CPlowerbound} tells us that $E$ does not have full 3-torsion.  Further, since $C_2(E) = 2$, by Proposition \ref{R-DZB}, it follows that $E(\Q^{ab})[2^{\infty}] \subseteq \Z/2\Z \times \Z/4\Z$. Therefore, by Lemma \ref{isogeny} we have that $E(\Q^{ab})_{\text{tors}} \subseteq \Z/2\Z \times \Z/28\Z$.

Finally if $C_3(E) = C_2(E) = 1$, then by Lemma \ref{isogeny}, Corollary \ref{CPlowerbound}, and Lemma \ref{2powertorsion} we have that $E(\Q^{ab})_{\text{tors}} \subseteq \Z/2\Z \times \Z/14\Z$.
\end{proof}

\begin{proposition}\label{prop:13isogeny}
Let $E/\Q$ be an elliptic curve. Suppose $E$ has a 13-isogeny. Then $E(\Q^{ab})_{\emph{tors}} \cong \Z/13\Z$ or $\Z/2\Z \times \Z/26\Z$.
\end{proposition}
\begin{proof}
Since there are no curves over $\Q$ with rational isogenies of degree properly divisible by 13, it follows from Lemma \ref{isogeny} that $E(\Q^{ab})_{\text{tors}} \cong \Z/m\Z \times \Z/m\Z \times \Z/13\Z$ for some $m \geq 1$. However, by Theorem \ref{8qisogs} we have that $C_p(E) = 1$ for all primes $p \neq 13$. Thus by Corollary \ref{CPlowerbound} and Lemma \ref{2powertorsion} we have that $m=1 \text{ or } 2$.
\end{proof}

Note that from here a quick count of the possible sizes of the torsion subgroups here along with Lemma \ref{isogimpliespoint} for the example of \lmfdbec{26569}{a}{1} having a point of order 163 over $\Q^{ab}$ is already enough to prove Corollary \ref{sizebound}.

\section{Eliminating Possible Torsion}\label{sec:class}

We restate the classification theorem for convenience:

\begin{theorem}
Let $E/\mathbb{Q}$ be an elliptic curve.  Then $E(\Q^{ab})_{\textup{tors}}$ is isomorphic to one of the following groups:
$$
\begin{array}{lr}
\Z/N_1\Z, & N_1 = 1, 3, 5, 7, 9, 11, 13, 15, 17, 19, 21, 25, 27, 37, 43, 67, 163, \\
\Z/2\Z \times \Z/2N_2\Z, & N_2 = 1, 2, \ldots, 9,\\
\Z/3\Z \times \Z/3N_3\Z, & N_3 = 1, 3,\\
\Z/4\Z \times \Z/4N_4\Z, & N_4 =  1, 2, 3, 4,\\
\Z/5\Z \times \Z/5\Z, & \\
\Z/6\Z \times \Z/6\Z, &  \\
\Z/8\Z \times \Z/8\Z. & \\
\end{array}
$$
Each of these groups appear as $E(\Q^{ab})_{\textup{tors}}$ for some elliptic curve over $\Q$.
\end{theorem}

We now eliminate the possibility of many of the groups appearing in the previous propositions as possible torsion subgroups over $\Q^{ab}$ for some elliptic curve $E/\Q$. We begin with a simple observation about 2-torsion over $\Q^{ab}$ from Proposition \ref{R-DZB}.

\begin{lemma}\label{lem:nosingle2torsion}
Let $E/\Q$ be an elliptic curve.  If $E(\Q^{ab})[2] \neq \{ \mathcal{O} \}$ then $E(\Q^{ab})[2] \cong \Z/2\Z \times \Z/2\Z$. Thus $E(\Q^{ab}){\emph{tors}} \not\cong \Z/2N\Z$ for any $N \geq 1$.
\end{lemma}

This eliminates many possible torsion structures over $\Q^{ab}$. In particular, after we combine the possibilities for $E(\Q^{ab})_{\text{tors}}$ from Propositions \ref{prop:13isogeny}, \ref{prop:7isogeny}, \ref{prop:5isogeny}, \ref{prop:3isogeny}, and \ref{prop:2isogeny}, and eliminate those groups ruled out by Lemma \ref{lem:nosingle2torsion}, we can compare them to the classification in Theorem \ref{classification} to see that it remains to rule out the following groups as possibilities for $E(\Q^{ab})_{\text{tors}}$:

$$
\begin{array}{lr}
\Z/2\Z \times \Z/2N_2\Z, & N_2 = 10, 12, 13, 14, 15, 18, 25, \\
\Z/3\Z \times \Z/27\Z, &\\
\Z/6\Z \times \Z/12\Z. & \\
\end{array}
$$

\begin{proposition}
Let $E/\Q$ be an elliptic curve, then $E(\Q^{ab})_{\textup{tors}}$ is not isomorphic to $\Z/2\Z \times \Z/28\Z$ or $\Z/2\Z \times \Z/30\Z$.
\end{proposition}
\begin{proof}
In the case $\Z/2\Z \times \Z/28\Z$ the curve has a 14-isogeny by Lemma  \ref{isogeny} of which there are only two possible isomorphism classes of curves given by the $j$-invariants $-3^3 5^3$ and $3^3 5^3 17^3$ (see for instance Table 4 of \cite{lozanorobledo1}).  Using division polynomials we can check that in both cases there are no points of order 4 defined over an abelian extension of $\Q$.

In the case $\Z/2\Z \times \Z/30\Z$ the curve has a 15-isogeny by Lemma \ref{isogeny}.  Here there are four possible $j$-invariants, $-\frac{5^2}{2}$, $-\frac{5^2\cdot 241^3}{2^3}$, $-\frac{5\cdot 29^3}{2^5}$, and $\frac{5\cdot 211^3}{2^{15}}$.  Again using division polynomials we can check that none of these curves have a point of order 2 defined over an abelian extension of $\Q$.
\end{proof}

\begin{proposition}\label{prop:noz2z26}
Let $E/\Q$ be an elliptic curve, then $E(\Q^{ab})_{\textup{tors}} \not\cong \Z/2\Z \times \Z/26\Z$.
\end{proposition}
\begin{proof}
Suppose that $E(\Q^{ab})_{\text{tors}} \cong \Z/2\Z \times \Z/26\Z$. Then $E$ has a 13-isogeny, and so by \cite{lozanorobledo1} Table 3 the curve has a $j$-invariant of the form:
\[ j(E) = \dfrac{(h^2+5h+13)(h^4+7h^3+20h^2+19h+1)^3}{h} \]
for some $h \in \Q$ with $h \neq 0$.  Thus, $E$ must be a twist of the curve
\[ E' : y^2 + xy = x^3 - \dfrac{36}{j(E)-1728}x - \dfrac{1}{j(E)-1728}. \]
Since $E(\Q^{ab})[2^{\infty}] \cong \Z/2\Z \times \Z/2\Z$ but $E$ does not have any 2-isogenies by Theorem \ref{8qisogs}, we must have $\Gal(\Q(E[2])/\Q) \cong \Z/3\Z$, implying that the discriminant of $E$ is a square. Since $E$ is a twist of $E'$, the discriminant of $E$ differs from the discriminant of $E'$ by at most a square.  Thus, we obtain a formula  $y^2 = \operatorname{Disc}(E')$, which we compute in terms of $h$.  By absorbing squares into the $y^2$ term we obtain a curve
\[C : Y^2 = h(h^2+6h+13)\]
which is a modular curve describing precisely when $E$ has a 13-isogeny and $\Gal(\Q(E[2])/\Q) \cong \Z/3\Z$.  This curve is actually an elliptic curve with $C(\Q) = \{ ( 0 : 0 : 1 ), (0 : 1 : 0) \} \cong \Z/2\Z$, both points being cusps.  Therefore there are no elliptic curves with $E(\Q^{ab}) \cong \Z/2\Z \times \Z/26\Z$.
\end{proof}

\begin{proposition}
Let $E/\Q$ be an elliptic curve, then $E(\Q^{ab})_{\textup{tors}} \not\cong \Z/2\Z \times \Z/50\Z$.
\end{proposition}
\begin{proof}
Suppose that $E(\Q^{ab})_{\text{tors}} \cong \Z/2\Z \times \Z/50\Z$. Then $E$ has a 25-isogeny, and so by \cite{lozanorobledo1} Table 3 the curve has a $j$-invariant of the form:
\[ j(E) = \dfrac{(h^{10}+10h^8+35h^6-12h^5+50h^4-60h^3+25h^2-60h+16)^3}{(h-1)(h^4 + h^3 + 6h^2 + 6h + 11)} \]
for some $h \in \Q$ with $h \neq 1$  By a similar argument made in Proposition \ref{prop:noz2z26} we have that the discriminant of $E$ must be a square. We again obtain a formula  $y^2 = \operatorname{Disc}(E)$ and by absorbing squares into the $y^2$ term we obtain a curve
\[C : Y^2 = h^7 + 9h^5+25h^3-11h^2+20h-44\]
which is a modular curve describing precisely when $E$ has a 25-isogeny and $\Gal(\Q(E[2])/\Q) \cong \Z/3\Z$.  This is a genus 3 hyperelliptic curve.

We can construct a map $\pi$ from $C$ to an elliptic curve $\tilde{C} : y^2 = x^3 + x^2 - x$ given by

\begin{equation}\label{eqn:z2z50}
(x : y : z) \mapsto (x^3 - x^2z + 4xz^2 - 4z^3: yz^2: x^2z - 2xz^2 + z^3).\tag{$\ast$}
\end{equation}

The curve $\tilde{C}$ has Cremona label \lmfdbec{20}{a}{2} and rank 0 and torsion isomorphic to $\Z/6\Z$. It has rational points
\[
\tilde{C}(\Q) = \{ (0 : 1 : 0), (0 : 0 : 1), (-1:-1:1), (1: -1: 1), (-1:1:1), (1:1:1)\}
\]
and we can use ($\ast$) to explicitly compute the preimage of each point under $\pi$ to see that the only rational points on $C$ are $(1 : 0 : 1)$ and $(0:1:0)$. Note that we have $h = \frac{X}{Z}$ for points $(X : Y : Z) \in C$. The first points corresponds to $h=1$, which is a zero of the denominator of $j(E)$, and the second point is the point at infinity, which corresponds to $h=\infty$, which is not a value we can consider. Thus, both of these points are cusps, and therefore there are no elliptic curves with $E(\Q^{ab}) \cong \Z/2\Z \times \Z/50\Z$.
\end{proof}

\begin{proposition}
Let $E/\Q$ be an elliptic curve, then $E(\Q^{ab})_{\textup{tors}} \not\cong \Z/2\Z \times \Z/20\Z$.
\end{proposition}

\begin{proof}
Suppose for the sake of contradiction that $E$ is an elliptic curve over $\Q$ such that $E(\Q^{ab})_{\textup{tors}} \cong \Z/2\Z \times \Z/20\Z$.  Then, $E$ has a $\Q$-rational 10-isogeny, and so by \cite{lozanorobledo1} Table 3 the curve has a
\[
j(E) = \dfrac{(h^6 - 4h^5 + 16h + 16)^3}{(h+1)^2(h-4)h}
\]
for some $h \in \Q$ with $h \neq -1,0,4$. Thus, $E$ must be a twist of the curve
\[ 
E' : y^2 + xy = x^3 - \dfrac{36}{j(E)-1728}x - \dfrac{1}{j(E)-1728}. 
\]
Moving the 2-torsion point to (0,0) yields a model of $E'$ of the form
\[
E' : y^2 = x^3 + b(h)x^2 + d(h)x
\]
for the rational functions
\[
b(h)=\dfrac{-9h^{12} + 72h^9 - 144h^3 - 144}{h^{12} - 8h^9 - 8h^3 - 8},
\]
\[d(h)= {\scriptstyle \frac{1296h^{27} - 19440h^{24} + 62208h^{21} + 124416h^{18} - 248832h^{15} -    622080h^{12} + 995328h^6 + 995328h^3 + 331776}{h^{36} - 24h^{33} + 192h^{30} -    464h^{27} - 720h^{24} + 2304h^{21} + 2112h^{18} + 5760h^{15} + 14400h^{12} +    11776h^9 + 12288h^6 + 12288h^3 + 4096}}.
\]
Note that $j(E) \neq 0,1728$ since $E$ has a 10-isogeny, and thus $E$ is a quadratic twist of $E'$. By Lemma \ref{lemma:twists} we have that $E'(\Q^{ab})_{\text{tors}} \cong E(\Q^{ab})_{\text{tors}} \cong \Z/2\Z \times \Z/20\Z$. Now, since $E('\Q^{ab})[2^{\infty}] \cong \Z/2\Z \times \Z/4\Z$, Lemma \ref{z2xz4xzn} tells us that either
\[d(h) \in\left(\Q^{\times}\right)^2 \hspace{1cm} \mbox{ or } \hspace{1cm} (b(h)^2 - 4d(h))d(h) \in \left(\Q^{\times}\right)^2.\]
Denote the 4-torsion point over $\Q^{ab}$ by $Q$.

Suppose $d(h) \in \left(\Q^{\times}\right)^2$.  We obtain a formula $Y^2 =  d(h)$ and by absorbing squares we obtain the curve
\[
C: Y'^2 = h^3 + h^2 + 4h + 4
\]
which is a modular curve describing precisely when $E$ has a 10-isogeny and $\Gal(\Q(x(Q))/\Q) \cong \Z/2\Z \times \Z/2\Z$. This is the elliptic curve with Cremona label \lmfdbec{20}{a}{1} and rational points 
\[
C(\Q) = \{(0:1:0), (0:-2:1), (0:2:1), (4:-10:1), (4:10:1), (-1:0:1)\} \cong \Z/6\Z.
\]
However, all of these points are cusps as they correspond to $h=0,-1,4$ which are all zeros of the denominator of $j(E)$.  Therefore there are no such elliptic curves.

Suppose instead that $(b(h)^2 - 4d(h))d(h) \in \left(\Q^{\times}\right)^2$.  Again we obtain a formula $Y^2 =  (b(h)^2 - 4d(h))d(h)$ and by absorbing squares we obtain the curve
\[
\hat{C}: Y'^2 = h^3 - 3h^2 - 4h
\]
which is a modular curve describing precisely when $E$ has a 10-isogeny and $\Gal(\Q(x(Q))/\Q) \cong \Z/4\Z$. This is an elliptic curve with  Cremona label \lmfdbec{40}{a}{1} and rational points
\[
\hat{C}(\Q) = \{(0 : 0 : 1), (0 : 1 : 0), (-1 : 0 : 1), (4 : 0 : 1) \} \cong \Z/2\Z \times \Z/2\Z.
\]
Again, all of these points are cusps as they correspond to $h=0, -1, 4$.  Therefore there are no such elliptic curves.  Thus we can conclude that no such a curve $E$ exists.
\end{proof}

\begin{proposition}
Let $E/\Q$ be an elliptic curve, then $E(\Q^{ab})_{\textup{tors}} \not\cong \Z/2\Z \times \Z/36\Z$.
\end{proposition}

\begin{proof}
Suppose for the sake of contradiction that $E$ is an elliptic curve over $\Q$ such that $E(\Q^{ab})_{\textup{tors}} \cong \Z/2\Z \times \Z/36\Z$.  Then, $E$ has a $\Q$-rational 18-isogeny, and so by \cite{lozanorobledo1} Table 3 the curve has a $j$-invariant of the form
\[
j(E) = \dfrac{(h^3-2)^3(h^9-6h^6-12h^3-8)^3}{h^9(h^3-8)(h^3+1)^2}
\]
for some $h \in \Q$ with $h \neq -1, 0, 2$. Thus, $E$ must be a twist of the curve 
\[ 
E' : y^2 + xy = x^3 - \dfrac{36}{j(E)-1728}x - \dfrac{1}{j(E)-1728}. 
\]
Moving the 2-torsion point to (0,0) yields a model of $E'$ of the form
\[
E' : y^2 = x(x^2 + b(h)x + d(h))
\]
for the rational functions
\[
b(h)=\dfrac{(h^3-2)(h^9 - 6h^6 - 12h^3 - 8)}{h^{12} - 8h^9 - 8h^3 – 8},
\]
\[d(h)= \dfrac{(h+1)(h^2-h+1)(h^3-2)^2(h^9 - 6h^6 - 12h^3 - 8)^2}{(h^6 - 4h^3 - 8)^2(h^12 - 8h^9 - 8h^3 - 8)^2}.
\]
Note that $j(E) \neq 0,1728$ since $E$ has an 18-isogeny, and thus $E$ is a quadratic twist of $E'$. By Lemma \ref{lemma:twists} we have that $E'(\Q^{ab})_{\text{tors}} \cong E(\Q^{ab})_{\text{tors}} \cong \Z/2\Z \times \Z/20\Z$. Now, since $E('\Q^{ab})[2^{\infty}] \cong \Z/2\Z \times \Z/4\Z$, Lemma \ref{z2xz4xzn} tells us that either
\[d(h) \in\left(\Q^{\times}\right)^2 \hspace{1cm} \mbox{ or } \hspace{1cm} (b(h)^2 - 4d(h))d(h) \in \left(\Q^{\times}\right)^2.\]
Denote the 4-torsion point over $\Q^{ab}$ by $Q$.

Suppose $d(h)\in\left(\Q^{\times}\right)^2$.  We obtain a formula $Y^2 =  d(h)$ and by absorbing squares we obtain the curve
\[
C: Y'^2 = h^3 + 1
\]
which is a modular curve describing precisely when $E$ has a 18-isogeny and $\Gal(\Q(x(Q))/\Q) \cong \Z/2\Z \times \Z/2\Z$. This is an elliptic curve with Cremona label \lmfdbec{36}{a}{1} and rational points
\[
C(\Q) = \{(0 : 1 : 0), (0 : 1 : 1), (0 : -1 : 1), (2 : 3 : 1), (2 : -3 : 1), (-1 : 0 : 1)\} \cong \Z/6\Z.
\]
However, all of these points are cusps as they correspond to $h=-1, 0, 2$. Therefore there are no such elliptic curves.

Suppose instead that $(b(h)^2 - 4d(h))d(h) \in \left(\Q^{\times}\right)^2$.  Again we obtain a formula $Y^2 =  (b(h)^2 - 4d(h))d(h)$ and by absorbing squares we obtain the curve
\[
\hat{C}: Y'^2 = h^7 - 7h^4 - 8h
\]
which is a modular curve describing precisely when $E$ has a 18-isogeny and $\Gal(\Q(x(Q))/\Q) \cong \Z/4\Z$. This is a genus 3 hyperelliptic curve with an automorphism $\varphi$ defined by
\[
(x : y : z) \mapsto (2x^4 - 10x^3z + 12x^2z^2 + 8xz^3 - 16z^4: 36yz^3 : x^4 - 8x^3z + 24x^2z^2 - 32xz^3 + 16z^4) 
\] 
and taking the quotient of $\hat{C}$ by $\varphi$ gives a map $\pi$ from $\hat{C}$ to an elliptic curve $\hat{C}_{\varphi} : y^2 = x^3 - x^2 + x$ given by
\begin{equation}\label{eqn:z2xz36}
(x : y : z) \mapsto (xz(x^2-xz-2z^2) : yz^3 : x^2(x+z)^2).\tag{$\ast$}
\end{equation}

The curve $\hat{C}_{\varphi}$ has Cremona label \lmfdbec{24}{a}{4} and has rational points
\[
\hat{C}_{\varphi}(\Q) = \{ (0 : 1 : 0), (0 : 0 : 1), (1 : 1 : 1), (1 : -1 : 1)\}.
\]
We can use ($\ast$) to explicitly compute the preimage of each point under $\pi$ to compute the rational points on $\hat{C}$. We find that $\hat{C}(\Q) = \{(-1:0:1), (0:0:1), (2:0:1), (0:1:0) \}$. These points correspond to $h=0,-1,2$, which are zeros of the denominator of $j(E)$ and so are cusps.  Thus we can conclude that no such a curve $E$ exists.
\end{proof}

\begin{proposition}
Let $E/\Q$ be an elliptic curve, then $E(\Q^{ab})_{\textup{tors}} \not\cong \Z/6\Z \times \Z/12\Z$.
\end{proposition}
\begin{proof}
Suppose that $E(\Q^{ab})_{\text{tors}} \cong \Z/6\Z \times \Z/12\Z$.  Then by Lemma \ref{full6torsion} we have that $\Gal(\Q(E[2])/\Q) \cong \Z/2\Z$, and thus $E$ has a single non-trivial point $P$ of order 2 over $\Q$.  Take a model of $E$ of the form $E : y^2 = x(x^2 + bx + c)$.  Then $E[2] = \{ \mathcal{O}, (0,0), (\alpha, 0), (\overline{\alpha}, 0) \}$ where $\alpha$ and $\overline{\alpha}$ are roots of $x^2 + bx + c$.  Let $F = \Q(\sqrt{\Delta_E}) = \Q(E[2])$.

Since $E(\Q^{ab})[2^{\infty}] \cong \Z/2\Z \times \Z/4\Z$ we obtain a point of order 4 over $\Q^{ab}$.  Suppose $(0,0)$ is halved in $E(\Q^{ab})_{\text{tors}}$.  By Lemma \ref{halvingpoint} we have that the point of order 4 is defined over $K = F(\sqrt{2b + 2\sqrt{b^2-4c}})$.  The field $K$ must be an abelian extension of $\Q$ and so $K$ is either cyclic or biquadratic, i.e. $\Gal(K/\Q) \cong \Z/4\Z$ or $\Z/2\Z \times \Z/2\Z$ respectively.

If $K$ is cyclic quartic, then, as discussed in the proof of Lemma 4.5 in \cite{chou}, $E$ has $j$-invariant 78608.  We can check that $\Phi_3(X,78608)$ has no rational roots, where $\Phi_3(X,Y)$ denotes the third classical modular polynomial, and so no curve with j-invariant 78608 has a 3-isogeny, contradicting that $E(\Q^{ab})[3^{\infty}] \cong \Z/3\Z \times \Z/3\Z$.

If $K$ is biquadratic, then $b^2-4c$ must be a square, but then $E$ has full 2-torsion over $\Q$, contradicting the assumption that $\Gal(\Q(E[2])/\Q) \cong \Z/2\Z$.

Suppose instead that $(\alpha,0)$ or $(\overline{\alpha}, 0)$ is halved over $\Q^{ab}$, without loss of generality suppose $(\alpha,0)$ is halved.  Then by Lemma \ref{halvingpoint} we have a point of order 4 defined over $K = F(\sqrt{-\alpha})$ which, as discussed in the proof of Lemma 4.5 in \cite{chou}, is not a cyclic extension.  Also, $K$ is not biquadratic unless $\alpha \in \Q$, again implying that $E$ has full 2-torsion over $\Q$, a contradiction. Thus, there is no such curve $E$.
\end{proof}

\begin{proposition}
Let $E/\Q$ be an elliptic curve, then $E(\Q^{ab})_{\textup{tors}} \not\cong \Z/2\Z \times \Z/24\Z$.
\end{proposition}

\begin{proof}
Suppose for the sake of contradiction that $E$ is an elliptic curve over $\Q$ such that $E(\Q^{ab})_{\textup{tors}} \cong \Z/2\Z \times \Z/24\Z$.  Then, $E$ has a $\Q$-rational 12-isogeny, and so by \cite{lozanorobledo1} Table 3 the curve has a $j$-invariant of the form
\[
j(E) = \dfrac{(h^2-3)^3 (h^6-9h^4 + 3h^2 - 3)^3}{h^4 (h^2-9)(h^2-1)^3}
\]
for some $h \in \Q$ with $h \neq 0, \pm 1, \pm3$. Thus, $E$ must be a twist of the curve 
\[ 
E' : y^2 + xy = x^3 - \dfrac{36}{j(E)-1728}x - \dfrac{1}{j(E)-1728}. 
\]
Moving the 2-torsion point to (0,0) yields a model of $E'$ of the form
\[
E' : y^2 = x(x^2 + b(h)x + d(h))
\]
for the rational functions
\[
b(h)=\dfrac{(h^2-3)(h^6-9h^4+3h^2-3)}{h^8-12h^6+30h^4-36h^2+9},
\]
\[d(h)= \dfrac{h^2 (h^2-3)^2 (h^6-9h^4+3h^2 - 3)^2}{(h^4-6h^2-3)^2 (h^8-12h^6 + 30h^4 - 36h^2 + 9)^2}.
\]
For ease of notation going forward, we will write $b = b(h)$, $d= d(h)$, and it should be understood that many of the following variables are functions of $h$. From the proof of Lemma \ref{z2xz4xzn} we see that the point $Q \in E(\Q^{ab})$ of order 4 must satisfy $2Q = (0,0)$. Let $\alpha$ and $\overline{\alpha}$ be roots of $x^2 + bx + d$ (over $\Q[h]$) so that $E': y^2 = x(x-\alpha)(x-\overline{\alpha})$. Then, from Lemma \ref{halvingpoint} we have (without loss of generality) that $Q$ has $x$-coordinate 

$$(\sqrt{0-0})(\sqrt{0-\alpha}) \pm (\sqrt{0-0})(\sqrt{0-\overline{\alpha}}) \pm (\sqrt{0-\alpha})(\sqrt{0-\overline{\alpha}}) + 0 = \pm \sqrt{\alpha \overline{\alpha}} = \pm \sqrt{d}$$

Suppose that the $x$-coordinate of $Q$ is $\sqrt{d}$. Since there is a point of order 8 in $E(\Q^{ab})$, there exists a point $R \in E(\Q^{ab})$ such that $2R = Q$.  Denote $$\alpha = \frac{-b + \sqrt{b^2 - 4d}}{2}$$ and $$\overline{\alpha} = \frac{-b - \sqrt{b^2 - 4d}}{2}$$ so that we have
$$E': y^2 = x(x-\alpha)(x-\overline{\alpha}).$$
For ease of notation we denote $\delta = \sqrt{d}$.  We can apply Lemma \ref{halvingpoint} again to deduce that since such an $R$ exists, we must have $\delta$, $\delta - \alpha$, and $\delta - \overline{\alpha}$ are all squares in $\Q^{ab}$. Notice that through some simplification we have that $(\delta - \alpha)(\delta - \overline{\alpha}) = (b+2\delta)\delta$ and so it suffices to prove that $\delta, (b+2\delta)\delta, \delta-\alpha$ are squares in $\Q^{ab}$.

For any $h \in \Q$ we have $\delta \in \Q$ and so clearly $\sqrt{\delta} \in \Q^{ab}$. Similarly, for all $h \in \Q$, we have that $(b+2\delta)\delta \in \Q$, and so $\sqrt{(b+2\delta)\delta} \in \Q^{ab}$.

To see when $\delta - \alpha$ is square in $\Q^{ab}$ we will find the minimal polynomial of $\delta - \alpha$ over $\Q$, and find when this defines an abelian extension of $\Q$. Notice that $\delta - \alpha = \frac{1}{2}(b - 2\delta - \sqrt{b^2 - 4d})$. Let $\xi = \sqrt{\delta - \alpha} = \sqrt{\dfrac{b-2\delta - \sqrt{b^2 - 4d}}{2}}$. The minimal polynomial of $\xi$ is:

$$f(X) = X^4 - (b+2\delta)X^2 + (b+2\delta)\delta. $$

Now, we apply Proposition \ref{galgrouptype} and see that $f$ defines an abelian extension of $\Q$ if and only if

$$(b+2\delta)\delta \in \left(\Q^{\times}\right)^{2}$$
or 
$$((b+2\delta)^2-4(b+2\delta)\delta)(b+2\delta)\delta \in \left(\Q^{\times}\right)^{2}$$
which, by absorbing squares is equivalent to 
$$(b-2\delta)\delta \in \left(\Q^{\times}\right)^{2}.$$
These yield the curves
\[
C_1: y^2 = h^3 - 2h^2 -3h
\;\;\;\text{ and }\;\;\;
C_2: y^2 = h^3 + 2h^2 -3h
\]
respectively.

Now it remains to classify all rational points on $C_1$ and $C_2$. These are curves with Cremona label \lmfdbec{48}{a}{1} and \lmfdbec{24}{a}{1} respectively and have rank 0 with rational points

$$C_1(\Q) = \{ (0 : 0 : 1), (0 : 1 : 0), (3 : 0 : 1), (-1 : 0 : 1) \}$$
and
$$C_2(\Q) = \{(0 : 0 : 1), (0 : 1 : 0), (-1 : -2 : 1), (3 : -6 : 1), (1 : 0 : 1), (3 : 6 :
1), (-1 : 2 : 1), (-3 : 0 : 1) \}.$$

Note that all of these points correspond to $h=0, \pm 1, \pm 3$, which are zeros of the denominator of $j(E)$ and hence are cusps. Therefore there are no curves over $\Q$ with $E(\Q^{ab})_{\text{tors}} \cong \Z/2\Z \times \Z/24\Z$.

\end{proof}

The following result comes from a paper by Bourdon and Clark \cite{bourdonclark}. The theorem applies broadly to any elliptic curve over $\overline{\Q}$ with complex multiplication, but we will use it to show specifically that the torsion subgroup $\Z/3\Z \times \Z/27\Z$ does not appear over $\Q^{ab}$. 

\begin{theorem}[\cite{bourdonclark}, Theorem 2.7]\label{OKmodstructure}
Let $E/\mathbb{C}$ be an $\mathcal{O}_{K}-CM$ elliptic curve, and let $M \subset E(\mathbb{C})$ be a finite $\mathcal{O}_{K}$-submodule.  Then:
\begin{enumerate}
\item[(a)] We have $M = E[\operatorname{ann}M]$.  Hence:
\item[(b)] $M \cong \mathcal{O}_K / (\operatorname{ann} M)$.
\item[(c)] $\# M = | \operatorname{ann}M |$.
\end{enumerate}
\end{theorem}

This gives us an understanding of $\mathcal{O}_K$-submodules of $E(\mathbb{C})$ for an elliptic curve with CM by the maximal order.  We use these results to prove the following proposition:

\begin{proposition}
Let $E/\Q$ be an elliptic curve, then $E(\Q^{ab})_{\textup{tors}} \not\cong \Z/3\Z \times \Z/27\Z$.
\end{proposition}
\begin{proof}
Suppose that $E(\Q^{ab})_{\text{tors}} \cong \Z/3\Z \times \Z/27\Z$.  Then by Lemma \ref{isogeny} the curve $E$ has a $9$-isogeny over $\Q$.  By Corollary \ref{CPlowerbound}, $E$ has an independent $3$-isogeny as well.  Thus we have the following isogeny graph:

$$E' \xleftarrow{\hspace*{1cm} 3 \hspace*{1cm}} E \xrightarrow{\hspace*{2cm} 9 \hspace*{2cm}} E''$$

Taking the dual isogeny also of degree 3 from $E'$ to $E$ and composing it with 9-isogeny from $E$ to $E''$ shows that $E'$ has a 27-isogeny.  The modular curve $X_0(27)$ has genus one, and there is a unique 27-isogeny class of elliptic curves up to isomorphism. Examining the 27-isogeny class shows that $E$ has CM by the maximal order of $K = \Q(\sqrt{-3})$.

Now, notice that $E(\Q^{ab})_{\text{tors}}$ is an $\mathcal{O}_K$-submodule of $E(\mathbb{C})$, since $K \subseteq \Q^{ab}$. Since the prime $p=3$ ramifies in $K$, there is a unique prime ideal $\mathfrak{p}$ of $\mathcal{O}_K$ with $|\mathfrak{p}| = 3$ and we have $(3) = \mathfrak{p}^2$. By Theorem \ref{OKmodstructure} (b) we have that $E[27] \cong \mathcal{O}_K / (3)^3 \cong \mathcal{O}_K / \mathfrak{p}^6$. Suppose $I$ is an ideal of $\mathcal{O}_K/\mathfrak{p}^6$. Then $\mathfrak{p}^6 \subseteq I$ so $I | \mathfrak{p}^6$ and therefore $I = \mathfrak{p}^b$ for some $0 \leq b \leq 6$ by the unique factorization of ideals into prime ideals. Thus, the $\mathcal{O}_K$-submodules of $E[27]$ are all of the form $\mathfrak{p}^b/\mathfrak{p}^6$ for some $0 \leq b \leq 6$. Moreover, the exponent of $\mathcal{O}_K/\mathfrak{p}^b$ is the smallest power of 3 contained in $\mathfrak{p}^b$. Since $(3)^d = \mathfrak{p}^{2d}$, this smallest power is $3^{\ceil*{\frac{b}{2}}}$.  Further, by Theorem \ref{OKmodstructure} (c) we have $\# \mathcal{O}_K / \mathfrak{p}^b = 3^b$, we deduce that
$$\mathcal{O}_K / \mathfrak{p}^b \cong_{\Z} \Z/3^{\floor*{\frac{b}{2}}}\Z \times \Z/3^{\ceil*{\frac{b}{2}}}\Z$$

Notice that since $E(\Q^{ab})[27]$ is an $\mathcal{O}_K$-submodule of $E[27]$, we have that $\floor*{\frac{b}{2}} = 1$, implying $b=2$ or $b=3$, but also $\ceil*{\frac{b}{2}} = 3$, implying $b = 5$ or $b=6$, a contradiction.  Thus no such curve exists.
\end{proof}

\section{Algorithm}
\label{sec:algorithm}

We can combine our results to form an explicit algorithm to compute $E(\Q^{ab})_{\text{tors}}$ for any elliptic curve $E/\Q$. Note that this algorithm only relies on the information about subgroups excluded from appearing as $E(\Q^{ab})_{\text{tors}}$ by sections \ref{sec:bounding} and \ref{sec:class}. See the author's website for the Magma code that implements this algorithm.

The algorithm uses Lemma \ref{isogimpliespoint} and Lemma \ref{lem:Epabelian} repeatedly, as well as Table \ref{table:RDZB}. Moreover, the algorithm works for any elliptic curve $E/\Q$ because we exhaustively deal with all isogeny graphs possible by Theorem \ref{8qisogs}. We denote $T := E(\Q^{ab})_{\text{tors}}$.

\begin{itemize}

\item Compute the isogeny graph of $E$. Let $I$ denote the degrees of the isogenies $E$ has. Let $N$ denote the largest value in $I$. Let $I_2$ and $T_2$ denote the 2-primary part of $I$ and $T$ respectively.

\item Lemma \ref{isogimpliespoint} and \ref{lem:Epabelian} show that:

	\begin{itemize}

	\item If $N = 11, 13, 15, 17, 19, 21, 25, 27, 37, 43, 67, \text{ or } 163$, then $T \cong \Z/N\Z$.

	\item If $N = 10, 14, 16, \text{ or } 18$ then $T \cong \Z/2\Z \times \Z/N\Z$
	
	\item If $I = [1,5,5]$ then $ T \cong \Z/5\Z \times \Z/5\Z$

	\item If $I = [1,3,3,9]$ then $T \cong \Z/3\Z \times \Z/9\Z$.

	\item If $I = [1,3,3]$ then $T \cong \Z/3\Z \times \Z/3\Z$. Note we cannot have $T \cong \Z/6\Z \times \Z/6\Z$ because Lemma \ref{full6torsion} would imply a 2-isogeny.

	\item If $I = [1,2,3,3,6,6]$ then $T \cong \Z/6\Z \times \Z/6\Z$.
	
	\end{itemize}

\item Table \ref{table:RDZB} shows that:

	\begin{itemize}

	\item If $I = [1,2,2,2,4,4,4,4]$ then $T \cong \Z/8\Z \times \Z/8\Z$.

	\item If $I = [1,2,2,2,4,4,8,8]$ then $T \cong \Z/4\Z \times \Z/16\Z$.

	\item If $I = [1,2,4,4,8,8,8,8]$ or $[1,2,2,2,4,4]$ then $T \cong \Z/4\Z \times \Z/8\Z$.

	\item If $I = [1,2,4,4,8,8,16,16]$ then $T \cong \Z/2\Z \times \Z/16\Z$.

	\item If $I = [1,2,4,4,8,8]$ then $T \cong \Z/2\Z \times \Z/8\Z$.
	
	\item If $I = [1,2,2,2,3,6,6,6]$ then $T_2 \cong \Z/4\Z \times \Z/4\Z$ and $3 \in I$ shows that $T \cong \Z/4\Z \times \Z/12\Z$.
	
	\item If $I = [1,2,2,2]$ then $T_2 \cong \Z/4\Z \times \Z/4\Z$.

	\end{itemize}
	
\item The remaining cases require extra steps beyond simply computing $I$, as $I$ does not uniquely determine $T$. 

	\begin{itemize}
	
	\item If $N = 1, 3, 5, 7, \text{ or } 9$, then compute $E[2]$ to check whether $\Q(E[2])$ is abelian or not. If not, then $T \cong \Z/N\Z$. If so, then $T \cong \Z/2\Z \times \Z2N\Z$.
	
	\item If $I_2 = [1,2,4,4]$ then compute $E[8]$ and $E[4]$ to determine whether $E$ has a point of order 8 over $\Q^{ab}$ and whether $\Q(E[4])$ is abelian. This distinguishes between the following three cases:
	
		\begin{itemize}

		\item $T_2 \cong \Z/2\Z \times \Z/4\Z$. If $3 \in I$, then $T \cong \Z/2\Z \times \Z/12\Z$. If $3 \not\in I$, then $T \cong \Z/2\Z \times \Z/4\Z$.
		
		\item $T_2 \cong \Z/2\Z \times \Z/8\Z$. This implies $T \cong \Z/2\Z \times \Z/8\Z$.
		
		\item $T_2 \cong \Z/4\Z \times \Z/4\Z$. If $3 \in I$, then $T \cong \Z/4\Z \times \Z/12\Z$. If $3 \not\in I$, then $T \cong \Z/4\Z \times \Z/4\Z$.
	
		\end{itemize}
		
	\item If $I_2 = [1,2]$ then we compute $E[4]$ to determine whether $E$ has a point of order 4 defined over $\Q^{ab}$. This distinguishes between the following two cases:
	
		\begin{itemize}
		
		\item $T_2 \cong \Z/2\Z \times \Z/2\Z$. If $3 \in I$, then $T \cong \Z/2\Z \times \Z/6\Z$. If $5 \in I$, then $T \cong \Z/2\Z \times \Z/10\Z$. If $I = I_2 = 1,2$ then $T \cong \Z/2\Z \times \Z/2\Z$.
		
		\item $T_2 \cong \Z/2\Z \times \Z/4\Z$. If $3 \in I$, then $T \cong \Z/2\Z \times \Z/12\Z$. If $3 \not\in I$, then $T \cong \Z/2\Z \times \Z/4\Z$.
		
		\end{itemize}
	
	\end{itemize}

\end{itemize}

\section{Examples}\label{sec:examples}

We first examine all examples of curves with an $n$-isogeny where $X_0(n)$ has finitely many non-cuspidal points over $\Q$ in Table \ref{table:finite}. We refer to Table 4 of \cite{lozanorobledo1} for the $j$-invariants. We give the torsion subgroup over $\Q^{ab}$, the $j$-invariant, the Cremona labels of the elliptic curves, and the Galois group of the field of definition of the abelian torsion. We then find examples for all the other torsion subgroups appearing in Theorem \ref{classification} in Table \ref{table:rest}, computing the torsion subgroup over $\Q^{ab}$ using the method described in Section \ref{sec:algorithm}.

\vfill

\renewcommand{\arraystretch}{1.1}
\begin{table}[h]
\centering
\caption{Torsion from $n$-isogenies with $X_0(n)$ genus > 0}\label{table:finite}
\scalebox{0.75}{
\begin{tabular}{|c||c|c|c|}\hline

$E(\Q^{ab})_{\text{tors}}$ & $j(E)$ & Cremona Label & $\Gal(\Q(E(\Q^{ab})_{\text{tors}})/\Q)$ \\ \hline

\multirow{6}{*}{$\Z/11\Z$} & \multirow{2}{*}{$-11 \cdot 131^3$} & \lmfdbec{121}{a}{1} & $\Z/10\Z$ \\ \cline{3-4}
& & \lmfdbec{121}{c}{2} & $\Z/5\Z$ \\ \cline{2-4}
& \multirow{2}{*}{$-2^{15}$} & \lmfdbec{121}{b}{1} & $\Z/5\Z$\\ \cline{3-4}
& & \lmfdbec{121}{b}{2} & $\Z/10\Z$\\ \cline{2-4}
& \multirow{2}{*}{$-11^2$} & \lmfdbec{121}{c}{1} & $\Z/10\Z$\\ \cline{3-4}
& & \lmfdbec{121}{a}{2} & $\Z/5\Z$\\ \hline

\multirow{8}{*}{$\Z/15\Z$} & \multirow{2}{*}{$-5^2 / 2$} & \lmfdbec{50}{a}{1} & $\Z/4\Z$ \\ \cline{3-4}
& & \lmfdbec{50}{b}{3} & $\Z/4\Z$ \\ \cline{2-4}
& \multirow{2}{*}{$-5^2 \cdot 241^3 / 2^3$} & \lmfdbec{50}{a}{2} & $\Z/2\Z \times \Z/4\Z$\\ \cline{3-4}
& & \lmfdbec{50}{b}{4} & $\Z/2\Z \times \Z/4\Z$\\ \cline{2-4}
& \multirow{2}{*}{$-5 \cdot 29^3 / 2^5$} & \lmfdbec{50}{a}{3} & $\Z/2\Z$\\ \cline{3-4}
& & \lmfdbec{50}{b}{1} & $\Z/2\Z$\\ \cline{2-4}
& \multirow{2}{*}{$5 \cdot 211^3 / 2^{15}$} & \lmfdbec{50}{a}{4} & $\Z/2\Z \times \Z/2\Z$\\ \cline{3-4}
& & \lmfdbec{50}{b}{2} & $\Z/2\Z$\\ \hline

\multirow{2}{*}{$\Z/17\Z$} & $-17^2 \cdot 101^3 / 2$ & \lmfdbec{14450}{p}{1} & $\Z/16\Z$ \\ \cline{3-4}
& $-17 \cdot 373^3 / 2^{17}$ & \lmfdbec{14450}{p}{2} & $\Z/8\Z$\\ \hline

\multirow{2}{*}{$\Z/19\Z$} & \multirow{2}{*}{$-2^{15} \cdot 3^3$} & \lmfdbec{361}{a}{1} & $\Z/9\Z$ \\ \cline{3-4}
& & \lmfdbec{361}{a}{2} & $\Z/18\Z$ \\ \hline

\multirow{8}{*}{$\Z/21\Z$} & \multirow{2}{*}{$-3^2 \cdot 5^6 / 2^3$} & \lmfdbec{162}{b}{1} & $\Z/3\Z$ \\ \cline{3-4}
& & \lmfdbec{162}{c}{2} & $\Z/6\Z$ \\ \cline{2-4}
& \multirow{2}{*}{$3^3 \cdot 5^3 / 2$} & \lmfdbec{162}{b}{2} & $\Z/6\Z$\\ \cline{3-4}
& & \lmfdbec{162}{c}{1} & $\Z/6\Z$\\ \cline{2-4}
& \multirow{2}{*}{$-3^2 \cdot 5^3 \cdot 101^3/2^{21}$} & \lmfdbec{162}{b}{3} & $\Z/6\Z$\\ \cline{3-4}
& & \lmfdbec{162}{c}{4} & $\Z/2\Z \times \Z/6\Z$\\ \cline{2-4}
& \multirow{2}{*}{$-3^3 \cdot 5^3 \cdot 383^3 / 2^7$} & \lmfdbec{162}{b}{4} & $\Z/2\Z \times \Z/6\Z$\\ \cline{3-4}
& & \lmfdbec{162}{c}{3} & $\Z/6\Z$\\ \hline

\multirow{2}{*}{$\Z/27\Z$} & \multirow{2}{*}{$-2^{15} \cdot 3 \cdot 5^3$} & \lmfdbec{27}{a}{2} & $\Z/18\Z$ \\ \cline{3-4}
& & \lmfdbec{27}{a}{4} & $\Z/9\Z$ \\ \hline

\multirow{2}{*}{$\Z/37\Z$} & $-7 \cdot 11^3$ & \lmfdbec{1225}{h}{1} & $\Z/12\Z$ \\ \cline{3-4}
& $-7 \cdot 137^3 \cdot 2083^3$ & \lmfdbec{1225}{h}{2} & $\Z/36\Z$ \\ \hline

\multirow{2}{*}{$\Z/43\Z$} & \multirow{2}{*}{$-2^{18} \cdot 3^3 \cdot 5^3$} & \lmfdbec{1849}{a}{1} & $\Z/21\Z$ \\ \cline{3-4}
& & \lmfdbec{1849}{a}{2} & $\Z/42\Z$ \\ \hline

\multirow{2}{*}{$\Z/67\Z$} & \multirow{2}{*}{$-2^{15} \cdot 3^3 \cdot 5^3 \cdot 11^3$} & \lmfdbec{4489}{a}{1} & $\Z/33\Z$ \\ \cline{3-4}
& & \lmfdbec{4489}{a}{2} & $\Z/66\Z$ \\ \hline
\
\multirow{2}{*}{$\Z/163\Z$} & \multirow{2}{*}{$-2^{18} \cdot 3^3 \cdot 5^3 \cdot 23^3 \cdot 29^3$} & \lmfdbec{26569}{a}{1} & $\Z/81\Z$ \\ \cline{3-4}
& & \lmfdbec{26569}{a}{2} & $\Z/162\Z$ \\ \hline
\multirow{4}{*}{$\Z/2\Z \times \Z/14\Z$} & \multirow{2}{*}{$-3^3 \cdot 5^3$} & \lmfdbec{49}{a}{1} & $\Z/6\Z$ \\ \cline{3-4}
& & \lmfdbec{49}{a}{3} & $\Z/6\Z$ \\ \cline{2-4}
& \multirow{2}{*}{$-3^3 \cdot 5^3 \cdot 17^3$} & \lmfdbec{49}{a}{2} & $\Z/2\Z \times \Z/6\Z$\\ \cline{3-4}
& & \lmfdbec{49}{a}{4} & $\Z/6\Z$ \\ \hline

\multirow{2}{*}{$\Z/3\Z \times \Z/9\Z$} & \multirow{2}{*}{$0$} & \lmfdbec{27}{a}{1} & $\Z/6\Z$ \\ \cline{3-4}
& & \lmfdbec{27}{a}{3} & $\Z/6\Z$ \\ \hline

\end{tabular}
}

\end{table}

\pagebreak

\renewcommand{\arraystretch}{1.1}
\begin{table}[h]
\centering
\caption{Examples of remaining torsion subgroups}\label{table:rest}
\centering
\begin{tabular}{|c||c|c|c|}\hline
$E(\Q^{ab})_{\text{tors}}$ & $j(E)$ & Cremona Label & $\Gal(\Q(E(\Q^{ab})_{\text{tors}})/\Q)$ \\ \hline

$\{ \mathcal{O} \}$ & $2^{12} \cdot 3^3 / 37$ & \lmfdbec{37}{a}{1} & $\{1\}$ \\ \hline
$\Z/3\Z$ & $2^{13} / 11$ & \lmfdbec{44}{a}{1} & $\{1\}$ \\ \hline
$\Z/5\Z$ & $- 1 / 2^5 \cdot 19$ & \lmfdbec{38}{b}{1} & $\{1\}$ \\ \hline
$\Z/7\Z$ & $3^3 \cdot 4^3 / 2^7 \cdot 13$ & \lmfdbec{26}{b}{1} & $\{1\}$ \\ \hline
$\Z/9\Z$ & $-3 \cdot 73^3 / 2^9$ & \lmfdbec{54}{b}{3} & $\{1\}$ \\ \hline
$\Z/13\Z$ & $-2^{12} \cdot 7 / 3$ & \lmfdbec{147}{b}{1} & $\Z/3\Z$ \\ \hline
$\Z/25\Z$ & $-2^{12} / 11$ & \lmfdbec{11}{a}{3} & $\Z/5\Z$ \\ \hline
$\Z/2\Z \times \Z/2\Z$ & $ - 5^6 / 3^2 \cdot 23$ & \lmfdbec{69}{a}{1} & $\Z/2\Z$ \\ \hline
$\Z/2\Z \times \Z/4\Z$ & $11^6 / 3 \cdot 5 \cdot 7$ & \lmfdbec{315}{b}{1} & $\Z/2\Z \times \Z/2\Z$ \\ \hline
$\Z/2\Z \times \Z/6\Z$ & $2^8 \cdot 7$ & \lmfdbec{196}{a}{1} & $\Z/6\Z$ \\ \hline
$\Z/2\Z \times \Z/8\Z$ & $12721^3 / 3 \cdot 5 \cdot 7 \cdot 11^2$ & \lmfdbec{3465}{e}{1} & $\left(\Z/2\Z\right)^3$ \\ \hline
$\Z/2\Z \times \Z/10\Z$ & $2161^3 / 2^{10} \cdot 3^5 \cdot 11$  & \lmfdbec{66}{c}{1} & $\Z/2\Z$ \\ \hline
$\Z/2\Z \times \Z/12\Z$ & $71^3 / 2^4 \cdot 3^3 \cdot 5$ & \lmfdbec{30}{a}{1} & $\Z/2\Z \times \Z/2\Z$ \\ \hline
$\Z/2\Z \times \Z/16\Z$ & $103681^3 / 3^4 \cdot 5$ & \lmfdbec{15}{a}{5} & $\Z/2\Z \times \Z/2\Z \times \Z/4$  \\ \hline
$\Z/2\Z \times \Z/18\Z$ & $-5^3 \cdot 1637^3/ 2^{18} \cdot 7$ & \lmfdbec{14}{a}{3} & $\Z/2\Z \times \Z/6\Z$ \\ \hline
$\Z/3\Z \times \Z/3\Z$ & $-2^{18} \cdot 7^{3} / 19^3$ & \lmfdbec{19}{a}{1} & $\Z/2\Z$ \\ \hline
$\Z/4\Z \times \Z/4\Z$ & $19^6 / 3^2 \cdot 5^2 \cdot 7^2$ & \lmfdbec{315}{b}{2} & $\left(\Z/2\Z\right)^{4}$ \\ \hline
$\Z/4\Z \times \Z/8\Z$ &$37^3 \cdot 109^3 / 2^4 \cdot 3^4 \cdot7^2$ & \lmfdbec{126}{b}{2} & $\left(\Z/2\Z\right)^{4}$ \\ \hline
$\Z/4\Z \times \Z/12\Z$ & $7^3 \cdot 127^3 / 2^2 \cdot 3^6 \cdot 5^2$ & \lmfdbec{30}{a}{2} & $\left(\Z/2\Z\right)^{4}$ \\ \hline
$\Z/4\Z \times \Z/16\Z$ & $241^3 / 3^2 \cdot 5^2$ & \lmfdbec{735}{e}{2} & $\left(\Z/2\Z\right)^{3}\times \Z/4\Z$ \\ \hline
$\Z/5\Z \times \Z/5\Z$ & $-2^{12} \cdot 31^3 / 11^5$ & \lmfdbec{11}{a}{1} & $\Z/4\Z$ \\ \hline
$\Z/6\Z \times \Z/6\Z$ & $5^3 \cdot 43^4 / 2^6 \cdot 7^3$ & \lmfdbec{14}{a}{1} & $\Z/2\Z \times \Z/2\Z$ \\ \hline
$\Z/8\Z \times \Z/8\Z$ & $13^3 \cdot 17^3 / 3^4 \cdot 5^4$ & \lmfdbec{735}{e}{4} & $\left(\Z/2\Z\right)^{5}$ \\ \hline

\end{tabular}

\end{table}


\clearpage

\begin{thebibliography}{99}
%












































































\bibitem{magma} W. Bosma, J. Cannon., C. Playoust, {\em The Magma algebra system. I. The user language}, J. Symbolic Comput., 24 (1998), 235-265.

\bibitem{bourdonclark} A. Bourdon, P. Clark, {\em Torsion points and galois representations on CM elliptic curves}, preprint; \url{http://alpha.math.uga.edu/~pete/Bourdon_Clark_Abbey_12_18.pdf}.


\bibitem{chou} M. Chou, {\em Torsion of rational elliptic curves over quartic Galois number fields}, J. Number Theory 160 (2016), 603-628.


\bibitem{conrad} K. Conrad, {\em Galois groups of cubics and quartics (not characteristic 2)}, expository paper available at \url{http://www.math.uconn.edu/~kconrad/blurbs/galoistheory/cubicquartic.pdf}.

\bibitem{3infinity} H. Daniels, \'{A}. Lozano-Robledo, F. Najman, A. Sutherland, {\em Torsion subgroups of rational elliptic curves over the compositum of all cubic fields}, submitted, to appear in Mathematics of Computation.

\bibitem{cubicclass} M. Derickx, A. Etropolski, J. Morrow, M. van Hoeij, D. Zureick-Brown, {\em Sporadic torsion on elliptic curves}, in preparation.

\bibitem{fujita} Y. Fujita, {\em Torsion subgroups of elliptic curves in elementary abelian 2-extensions of $\Q$}, J. Number Theory 114 (2005), 124-134.


\bibitem{gonzlozano} E. Gonz\'{a}lez-Jim\'{e}nez, \'{A}. Lozano-Robledo, {\em Elliptic curves with abelian division fields}, Math. Z. 283 (2016), Volume 283, Issue 3, pp. 835–859.

\bibitem{gonznajman} E. Gonz\'{a}lez-Jim\'{e}nez, F. Najman, {\em Growth of torsion groups of elliptic curves upon base change}, preprint.





\bibitem{kamienny} S. Kamienny, {\em Torsion points on elliptic curves and q-coefficients of modular forms}, Invent. Math. 109 (1992), 221-229.

\bibitem{kenku} M.A. Kenku {\em On the number of $\Q$-isomorphism classes of elliptic curves in each $\Q$-isogeny class}, J. Number Theory 15 (1982), 199-202.

\bibitem{kenkumomose} M.A. Kenku, F. Momose, {\em Torsion points on elliptic curves defined over quadratic fields}, Nagoya Math. J. 109 (1988), 125-149.

\bibitem{knapp} A.W. Knapp, {\em Elliptic Curves}, Princeton University Press, Princeton, NJ, 1992.


\bibitem{lmfdb} The LMFDB Collaboration, {\em The L-functions and Modular Forms Database}, \url{http://www.lmfdb.org}, Online; accessed 16 September 2013.

\bibitem{lozanorobledo1} \'A. Lozano-Robledo, {\em On the field of definition of $p$-torsion points on elliptic curves over the rationals}, Mathematische Annalen, Vol 357, Issue 1 (2013), 279-305.


\bibitem{mazur1} B. Mazur, {\it Rational isogenies of prime degree}, Inventiones Math. 44 (1978), pp. 129 - 162.

\bibitem{najman} F. Najman, {\em Torsion of rational elliptic curves over cubic fields and sporadic points on $X_1(n)$}, Math. Res. Letters, 23 (2016) 245-272.

\bibitem{ribet} K. Ribet, {\em Torsion points of abelian varieties in cyclotomic extensions}, Enseign. Math. 27, pp. 315-319 (1981).

\bibitem{2adicimage} J. Rouse, D. Zureick-Brown, {\em Elliptic curves over $\mathbb{Q}$ and 2-adic images of galois}; arXiv:1402.5997v2 [math.NT]


\end{thebibliography}
\end{document}